\documentclass[12pt, reqno]{amsart}
\usepackage{esint}
\usepackage{mathrsfs}
\usepackage{amssymb,amsthm,amsmath}
\usepackage[numbers,sort&compress]{natbib}
\usepackage{amssymb,amsmath}
\usepackage{amsfonts}
\usepackage{mathrsfs}
\usepackage{latexsym}
\usepackage{amssymb}
\usepackage{amsthm}
\usepackage{color}
\usepackage{pdfsync}
\usepackage{indentfirst}
\usepackage{hyperref}
\date{June 18, 2021}

\usepackage{amsmath}
\usepackage{amsthm}

\newtheorem{theorem}{Theorem}

\newtheorem{remark}[theorem]{Remark}
\newtheorem{proposition}[theorem]{Proposition}
\newtheorem{lemma}[theorem]{Lemma}
\newtheorem{definition}[theorem]{Definition}
\newtheorem{corollary}[theorem]{Corollary}

\newcommand{\beq}{\begin{equation}}
\newcommand{\eeq}{\end{equation}}
\newcommand{\ben}{\begin{eqnarray}}
\newcommand{\een}{\end{eqnarray}}
\newcommand{\beno}{\begin{eqnarray*}}
\newcommand{\eeno}{\end{eqnarray*}}

\numberwithin{equation}{section}

\begin{document}
\title[Partial Regularity for the 3D chemotaxis-Navier-Stokes equations]{\bf  Hausdorff measure for the singularity set of the 3D chemotaxis-Navier-Stokes equations}
\author{Xiaomeng~Chen}
\address[Xiaomeng~Chen]{School of Mathematical Sciences, Dalian University of Technology, Dalian, 116024,  China}
\email{cxm@mail.dlut.edu.cn}

\author{Shuai~Li}
\address[Shuai~Li]{School of Mathematical Sciences, Dalian University of Technology, Dalian, 116024,  China}
\email{leeshy@mail.dlut.edu.cn}

\author{Lili~Wang}
\address[Lili~Wang]{School of Mathematical Sciences, Dalian University of Technology, Dalian, 116024,  China}
\email{wanglili\_@mail.dlut.edu.cn}

\author{Wendong~Wang}
\address[Wendong~Wang]{School of Mathematical Sciences, Dalian University of Technology, Dalian, 116024,  China}
\email{wendong@dlut.edu.cn}
\date{\today}
\maketitle
\begin{abstract}
Suspensions of aerobic bacteria often develop flows from the interplay of chemotaxis and buoyancy, which is so-called the chemotaxis-Navier-Stokes flow.  In 2004, Dombrowski et al. observed that Bacterial flow in a sessile drop related to those in the Boycott effect of sedimentation can carry bioconvective plumes, viewed from
below through the bottom of a petri dish, and the horizontal ``turbulence'' white line near the top is the air-water-plastic contact line. In pendant drops such self-concentration occurs at the bottom. On scales much larger than a cell, concentrated regions  exhibit  transient, reconstituting, high-speed jets straddled by vortex streets. It's interesting to verify these turbulent phenomena mathematically. In this note, we investigate the  Hausdorff dimension of these vortices (singular points) by considering partial regularity of weak solutions of the three dimensional  chemotaxis-Navier-Stokes equations, and showed that the singular dimension is not larger than $1$, which seems to  be consistent with the  linear singularity in the experiment.

\end{abstract}


\section{Introduction}	

There is a long research history on the Hausdorff measure of the singularity set of weak solutions to certain fluid models. As is well-known, the set of singular points of Leray-Hopf weak solutions of Navier-Stokes equations, 
has been widely studied.  For the suitable weak solutions (a subset of Leray-Hopf weak solutions) of Navier-Stokes equations, it was started by Scheffer in \cite{SV1,SV2,SV4}, and later Caffarelli-Kohn-Nirenberg \cite{CKN} showed that the set $\mathcal{S}$ of possible interior singular points of a suitable weak solution is one-dimensional parabolic Hausdorff measure zero. These partial regularity results have an interesting consequence in the
context of the experimental study of {\it turbulence}, since one can  relate the singular
set of a flow to its turbulence region (see \cite{Tsai-2018}).

The turbulence phenomena also happened in the experiment of chemotaxis fluids.
Consider a PDE model on $Q_T=\mathbb{R}^3\times(0,T)$ describing the dynamics of oxygen,
swimming bacteria, and viscous incompressible fluids, which was proposed by Tuval et al.\cite{ILCJR 2005} as follows:
\begin{eqnarray}\label{eq:GKS}
 \left\{
    \begin{array}{llll}
    \displaystyle \partial_t n+u\cdot \nabla n-\Delta n=-\nabla\cdot(\chi(c)n\nabla c),\\
    \displaystyle \partial_t c+ u\cdot \nabla c-  \Delta c=-\kappa(c)n, \\
    \displaystyle \partial_t u+ u\cdot \nabla u-\Delta u+\nabla P =-n\nabla \phi,~~\nabla\cdot u=0,
    \end{array}
 \right.
\end{eqnarray}
where $\mathbb{R^{+}}=(0,+\infty)$, $c(x,t):Q_T\rightarrow{\mathbb R}^{+}$, $n(x,t):Q_T\rightarrow\mathbb{R}^{+}$, $u(x,t):Q_T\rightarrow\mathbb{R}^{3}$ and $P(x,t):Q_T\rightarrow\mathbb{R}$ denote the oxygen concentration, cell concentration, the fluid velocity and the associated pressure, respectively. Moreover, the
gravitational potential $\phi$, the chemotactic sensitivity $\chi(c)\geq 0$ and the per-capita oxygen
consumption rate $\kappa(c)\geq 0$ are supposed to be sufficiently smooth given functions. 
Dombrowski et al. observed in \cite{DCCGK2004} (see also \cite{ILCJR 2005}) in the experiment: Bacterial ``turbulence'' in a sessile drop lies in the air-water-plastic contact line. The central
fuzziness is due to collective motion, not quite captured at the
frame rate of $\frac{1}{30}$ s. In pendant drops, a fluctuation increasing
the local concentration leads to a jet descending faster
than its surroundings, which entrains nearby fluid to
produce paired, oppositely signed vortices.
  It's interesting to verify these turbulent phenomena mathematically. Motivated by partial regularity theory of Caffarelli-Kohn-Nirenberg \cite{CKN}, global suitable weak solution was constructed in a previous paper of the authors \cite{CLWW2023}, and this article is aimed to describe the properties of singularities. 

First, let us briefly review some well-posed results for the system (\ref{eq:GKS}). Global classical solutions near constant steady states were constructed for the full
chemotaxis-Navier-Stokes system by Duan-Lorz-Markowich in \cite{DLM2010} with small data. In \cite{AL2010}, for the case of bounded domain of $\mathbb{R}^n$ with $n= 2,3$, the local existence of weak solutions for problem (\ref{eq:GKS}) is obtained by Lorz. By assuming $\chi',\kappa'\geq 0$ and $\kappa(0)=0$, local well-posed results and blow-up criteria were established by Chae-Kang-Lee in \cite{CKL2013}. 
For the two-dimensional system  of (\ref{eq:GKS}), the system is better understood. Liu and Lorz \cite{LL2011}  proved the global existence of weak solutions to the two-dimensional system  of (\ref{eq:GKS}) for arbitrarily large initial data, under the assumptions on $\chi$ and $f$ made in \cite{DLM2010}.
For more developments, we refer to \cite{Winkler2012, CKL2014,Winkler2016,HZ2017,WWX2018,KLW2022,WWX2018-2,WZZ2021} and the references therein. 

Second, global weak solutions of Leray-Hopf type for this system were obtained in 2D and 3D by Zhang-Zheng \cite{ZZ2014}, He-Zhang  \cite{HZ2017}, and   Kang-Lee-Winkler \cite{KLW2022}, respectively, where they established a priori estimate
\ben\label{eq: a priori-zhang}
\mathcal{U}(t) + \int_0^t \mathcal{V}(t) d\tau \leq C e^{Ct},
\een
where
\beno
\mathcal{U} = \|n\|_{L^1 \cap L \log L} + \|\nabla \sqrt c\|_{L^{2}}^2 + \|u\|_{L^{2}}^2,
\eeno
and
\beno
\mathcal{V} = \|\nabla \sqrt{n+1}\|_{L^{2}}^2 + \|\Delta \sqrt c\|_{L^{2}}^2 + \|\nabla u\|_{L^{2}}^2 + \int_{\mathbb{R}^d} (\sqrt{c})^{-2} |\nabla \sqrt c|^4 dx + \int_{\mathbb{R}^d} n |\nabla \sqrt c|^2 dx,
\eeno
where $d=2,3$.
However, up to now more information about these weak solutions is still unknown,  especially for the interior singular vortices as described in \cite{DCCGK2004} or the self-organized generation of a persistent hydrodynamic vortex that traps cells near the contact line (see \cite{ILCJR 2005}).
Motivated by  \cite{CKN} and \cite{DD},
partial regularity of local strong solutions was first investigated in \cite{CLW 2022} by Chen-Li-Wang, where they considered the simplified 3D chemotaxis-Navier-Stokes equations ($\kappa(c)=c$, $\chi(c)=1$ ) at the first blow-up time and obtained  the possible singular set has zero $\frac53$-dimensional Hausdorff measure. For the general system $\eqref{eq:GKS}$, 
global suitable weak solutions were constructed under the following certain assumptions about $ \chi $ and $\kappa$ in \cite{CLWW2023}:
\ben\label{ine:chi}
\chi(s) \in C^2(\overline{\mathbb{R^{+}}});~~~~~\chi(s)\geq 0;
\een
\ben\label{ine:kappa}
\kappa(s) \in C^2(\overline{\mathbb{R^{+}}}),~~~~\kappa(0)=0,~~~~ \kappa'(s)\geq 0,~~~~\kappa''(s)\geq 0;
\een
and
\ben\label{equ:chi kappa}
\kappa(s)=\Theta_0s\chi(s),
\een
with $ \Theta_0 >0$ is a positive absolute constant. In details, the existence theorem is stated as follows:
\begin{theorem}[Theorem 1.2 in \cite{CLWW2023}]\label{suitable weak solution}
	Assume that the initial data  $(n_0,c_0,u_0)$ satisfies
	\ben\label{ine:initial condition assumption 1}
	\left\{
	\begin{array}{llll}
		\displaystyle n_0 \in L^1(\mathbb{R}^3), \quad (n_0+1) \ln (n_0+1)  \in L^1(\mathbb{R}^3),  \quad u_0 \in L^2_\sigma(\mathbb{R}^3);\\
		\displaystyle \nabla \sqrt{c_0} \in L^2(\mathbb{R}^3), \quad c_0 \in L^1 \cap L^\infty(\mathbb{R}^3);\\
		\displaystyle n_0 \geq 0, \quad c_0 \geq 0.\\
	\end{array}
	\right.
	\een
	Moreover, $\nabla \phi\in L^\infty(\mathbb{R}^3)$, $\kappa$ and $\chi$ satisfy $\eqref{ine:chi}$, $\eqref{ine:kappa}$ and $ \eqref{equ:chi kappa}. $
	Then there exists a global suitable weak solution of the system (\ref{eq:GKS}). 
\end{theorem}


The  suitable weak solutions is defined as follows:
\begin{definition}[Definition 1.1 in \cite{CLWW2023}]\label{sws}
	A triplet $(n, c, u)$ is called a suitable weak solution of the system (\ref{eq:GKS}) with the initial data satisfies (\ref{ine:initial condition assumption 1}) in $\mathbb{R}^3\times (0,T)$, if the following holds:\\
	
	(i) For any bounded domain $\Omega \subset \mathbb{R}^3$, \\
	
	$n, n \ln n \in L^\infty_{\rm loc}((0,T);L^1(\Omega))$, $\nabla \sqrt{n} \in L^2_{\rm loc}((0,T);L^2(\Omega))$, \\
	
	$\nabla \sqrt{c} \in L^\infty_{\rm loc}((0,T);L^2(\Omega)) \cap L^2_{\rm loc}((0,T);H^1(\Omega))$, \\
	
	$u\in L_{\rm loc}^\infty((0,T);L^2(\Omega)) \cap L^2_{\rm loc}((0,T);H^1(\Omega))$, $P \in L^\frac32_{\rm loc}((0,T);L^\frac32(\Omega))$; \\
	
	(ii) $(n, c, u)$ solves (\ref{eq:GKS}) 
	in the sense of distributions;\\
	
	(iii) For any $t<T$, $(n, c, u)$ satisfies the following energy inequality: 
	\begin{equation*}
		\begin{aligned}
		&||u||_{L^{2}}^{2} + \int_0^t||\nabla u(t)||_{L^{2}}^{2}\\&+~\int_{\mathbb{R}^3} (n+1) \ln (n+1)(\cdot,t) + \int_0^t\int_{\mathbb{R}^3} |\nabla \sqrt{n+1}|^2\\
		&+~ \frac2{\Theta_0}||\nabla \sqrt{c}||_{L^{2}}^{2} +\frac4{3\Theta_0}\int_0^t||\nabla^2 \sqrt{c}||_{L^{2}}^{2} + \frac1{3\Theta_0} \int_0^t \int_{\mathbb{R}^3}(\sqrt{c})^{-2} |\nabla \sqrt{c}|^4  \\
		\leq ~&
		C(\|\nabla \phi\|_{L^\infty},\|n_0\|_{L^1},\|c_{0}\|_{L^{\infty}\cap L^{1}},\|u_{0}\|_{L^{2}},\|(n_{0}+1)\ln (n_{0}+1)\|_{L^{1}},\|\nabla \sqrt{c_0}\|_{L^{2}})(1+t);
		\end{aligned}
	\end{equation*}
	(iv) For any $t<T$, $(n, c, u)$ satisfies the local energy inequality:
\begin{equation}\label{local energy inequality}
\begin{aligned}
&\int_{\Omega} (n \ln n \psi)(\cdot,t) + 4 \int_{(0,t)\times\Omega} |\nabla \sqrt{n}|^2 \psi+\frac{2}{\Theta_0}  \int_{\Omega} (|\nabla \sqrt{c}|^2 \psi)(\cdot,t)\\&+\frac{4}{3\Theta_0}\int_{(0,t)\times\Omega} |\Delta \sqrt{c}|^2 \psi
+\frac{18}{\Theta_0}\|{c_{0}}\|_{L^{\infty}}\int_{\Omega}(|u|^2)(\cdot,t) \psi\\
& + \frac{18}{\Theta_0}\|{c_{0}}\|_{L^{\infty}}\int_{(0,t)\times\Omega} |\nabla u|^2 \psi+\frac{2}{3\Theta_0}  \int_{(0,t)\times\Omega} (\sqrt{c})^{-2} |\nabla \sqrt{c}|^4 \psi
\\\leq~&\int_{(0,t)\times\Omega} n \ln n (\partial_t \psi + \Delta \psi) + \int_{(0,t)\times\Omega} n \ln n u \cdot \nabla \psi \\&+\int_{(0,t)\times\Omega}n\chi(c)\nabla c\cdot \nabla \psi +\int_{(0,t)\times \Omega}n\ln n \chi(c)\nabla c\cdot\nabla\psi\\&+ \frac{2}{\Theta_0}\int_{(0,t)\times\Omega} |\nabla \sqrt{c}|^2 (\partial_t \psi + \Delta \psi) + \frac{2}{\Theta_0} \int_{(0,t)\times\Omega} |\nabla \sqrt{c}|^2 u \cdot \nabla \psi\\&
+\frac{18}{\Theta_0}||c_{0}||_{L^{\infty}}\int_{(0,t)\times\Omega} |u|^2 \left(\partial_t \psi +\Delta \psi\right) + \frac{18}{\Theta_0}||c_{0}||_{L^{\infty}}\int_{(0,t)\times\Omega} |u|^2 u\cdot\nabla\psi
\\&+\frac{36}{\Theta_0}||c_{0}||_{L^{\infty}}\int_{(0,t)\times\Omega} (P - \bar{P}) u \cdot \nabla \psi - \frac{36}{\Theta_0}||c_{0}||_{L^{\infty}}\int_{(0,t)\times\Omega} n\nabla\phi \cdot u \psi,
\end{aligned}
\end{equation}
	where  $\psi\geq 0$ and vanishes in the parabolic boundary of $(0,t)\times\Omega.$
\end{definition}

\begin{remark}
We remark that the above local energy inequality is not unique, since one can obtain more many inequalities by relaxing some constants of coupling estimates. The first term on the left hand side may be not positive, which is different from the local energy inequality of Navier-Stokes equations. Hence, (\ref{local energy inequality}) is a local energy inequality of weak form.
\end{remark}

For $Q_{r}(z_0)=B_{r}(x_0)\times (t_0-r^2,t_0)$, where $z_0=(x_0,t_0)$, we say the solution $(n,c,u)$ of (\ref{eq:GKS}) is regular at $z_0$ if there exists $r_1>0$ such that 
$(n,\nabla c,u)\in L^\infty(Q_{r_1}(z_0))$. Moreover, let
\beno
\|\chi\|_0:=\|\chi\|_{L^\infty(0, \|c_0\|_{L^{\infty}})}+\|\chi'\|_{L^\infty(0, \|c_0\|_{L^{\infty}})}+\|\chi''\|_{L^\infty(0, \|c_0\|_{L^{\infty}})},
\eeno
and
\beno
\|\kappa\|_0:=\|\kappa\|_{L^\infty(0, \|c_0\|_{L^{\infty}})}+\|\kappa'\|_{L^\infty(0, \|c_0\|_{L^{\infty}})}
+\|\kappa''\|_{L^\infty(0, \|c_0\|_{L^{\infty}})}.
\eeno
Our main results are stated  as follows.
\begin{theorem}\label{thm:3} Assume that $(n,c,u)$ is a suitable weak solution of (\ref{eq:GKS}) in $\mathbb{R}^3\times(0,T)$, $Q_{r}(z_0)\subset  \mathbb{R}^3\times(0,T)$ and $0<\delta_0\leq \frac{1}{10}$. Then $z_0=(x_0,t_0)$ is a regular point, if there exists a constant $\varepsilon_1$ depending on $\delta_0$
and $\Theta_0$	such that
	\ben\label{assume:3}\nonumber
	&&\limsup_{r \to 0} r^{-1-\delta_0} \int_{Q_{r}(z_0)} |\nabla \sqrt{n}|^2 + \limsup_{r \to 0} r^{-1} \left(\int_{Q_{r}(z_0)} |\nabla u|^2 + |\nabla^2 \sqrt{c}|^2\right) \\&\leq& \frac{\varepsilon^8_1}{625(1+\|\chi\|_{0})^{80}(1+\|\nabla\phi\|_{L^\infty}+\|c_0\|_{L^\infty})^{160}}.
	\een
\end{theorem}

Consequently,  the Hausdorff measure for the set of singularity points follows naturally.
\begin{corollary}\label{thm:4}
	Under the assumptions of Theorem \ref{thm:3}, we have
	\beno
	\mathcal{P}^{1+\delta_0}(\mathcal{S}) = 0,
	\eeno
	where $\mathcal{S}$ is the singular set  of (\ref{eq:GKS}) and $\mathcal{P}^\alpha(\mathcal{S})$ is the Hausdorff measure of  parabolic version.
\end{corollary}

\begin{remark} The singular set  of (\ref{eq:GKS}) has $1^+$-dimensional Hausdorff measure, which seems to be  consistent with the observation in the experiment of \cite{DCCGK2004}, since the  ``turbulence" happened in the air-water-plastic contact line. When $n,c$ vanishes, the above theorem is also similar as the Navier-Stokes case (see \cite{CKN}). It's interesting that whether the dimensional exponent $1^+$ can be improved to $1$,  and it is difficult to improve it to $1^-$, since it is greatly open even for the Navier-Stokes equations.
\end{remark}
\begin{remark}
The system (\ref{eq:GKS}) has the following scaling property as the  Navier-Stokes equations: if $(n,c,u,p)$ is a solution, then for any $\rho_0>0$,
\ben\label{eq:scaling'}
&&n_{\rho_0}(x,t)=\rho_0^2 n(\rho_0 x, \rho_0^2t);~~ c_{\rho_0}(x,t)=c(\rho_0 x, \rho_0^2t),\nonumber\\
&&~~ u_{\rho_0}(x,t)=\rho_0 u(\rho_0 x, \rho_0^2t);~~p_{\rho_0}(x,t)=\rho_0^2 p(\rho_0 x, \rho_0^2t),
\een
$(n_{\rho_0},c_{\rho_0},u_{\rho_0},p_{\rho_0})$ is also a solution. The main difficulty lies in the first term of the local energy inequality (\ref{local energy inequality}), which is not scaling invariant.
\end{remark}

The above results are based on the following regularity criteria.
\begin{theorem}\label{thm:1}
Assume that $(n,c,u)$ is a suitable weak solution of (\ref{eq:GKS}) in $\mathbb{R}^3\times(-1,0)$, then $z_0=(x_0,0)$ is a regular point, if there exists a constant $\varepsilon_1$, which depends on $ \Theta_0$
such that one of the following conditions holds \\
\par
$(i)$
\ben\label{eq:condition}
&&\sup_{-1<t<0}\int_{B_1(x_0)} \left(n + |n \ln n| + |\nabla \sqrt{c}|^2 + |u|^2\right)dx \\
&&+ \int_{Q_1(z_0)} \left(|\nabla \sqrt{n}|^2 + |\nabla u|^2 + |\nabla^2 \sqrt{c}|^2 + |P|^\frac32\right)dxdt \nonumber\\
&&\leq \varepsilon_0:= \frac{\varepsilon_1}{(1+\|\chi\|_{0})^{12}(1+\|c_0\|_{L^\infty}+\|\nabla \phi\|_{L^\infty})^{24}}\nonumber;
\een	
\par$(ii)$\ben\label{eq:condition-32}
  &&\int_{Q_{1}(z_0)} \left(n^{\frac32}(|\ln n|+1)^\frac32 + |\nabla \sqrt{c}|^3 + |u|^3+|P|^\frac32\right)dxdt \nonumber\\
  &&\leq \frac{\varepsilon_1^{2}}{(1+\|\chi\|_{0})^{20}\left( 1+\|\nabla\phi\|_{L^{\infty}}+\|c_{0}\|_{L^{\infty}} \right)^{40}}.
\een
\end{theorem}


The paper is organized as follows.  The proof of Theorem \ref{thm:3} is given in Section 2 under the assumption of Theorem \ref{thm:1}. Hausdorff measure of the set of singularities is stated in Section 3. Section 4 aims to prove Theorem \ref{thm:1} by method of mathematical induction. Besides, some fundamental lemmas are presented in the appendix.

Throughout this article, $C(A,B)$ denotes an absolute constant of depending on $A,B$ but independent of $(n,c,u)$ and may be different from line to line.
We write $ L^{p}(\mathbb{R}^{3})=L^{p} $ and $\|f\|_{L^{p}}=\|f\|_{p}$ for simplicity.

\section{Proof of Theorem \ref{thm:3}}

In this section, we will prove Theorem \ref{thm:3} under the assumption of Theorem \ref{thm:1}.
For convenience, let us introduce some invariant quantities under the scaling (\ref{eq:scaling'}):
\ben\label{A,E,C,D}\nonumber
&&A_u(r,z_0)=r^{-1}\|u\|^2_{L^\infty_tL^2_x(Q_r(z_0))};\qquad\qquad
E_u(r,z_0)=r^{-1}\|\nabla u\|^2_{L^2_tL^2_x(Q_r(z_0))};\\\nonumber
&&A_{\nabla\sqrt{c}}(r,z_0)=r^{-1}\|\nabla \sqrt{c}\|^2_{L^\infty_tL^2_x(Q_r(z_0))};\quad
E_{\nabla\sqrt{c}}(r,z_0)=r^{-1}\|\nabla^2 \sqrt{c}\|^2_{L^2_tL^2_x(Q_r(z_0))};\\\nonumber
&&A_{\sqrt{n}}(r,z_0)=r^{-1}\|\sqrt{n}\|^2_{L^\infty_tL^2_x(Q_r(z_0))};\qquad
E_{\sqrt{n}}(r,z_0)=r^{-1}\|\nabla \sqrt{n}\|^2_{L^2_tL^2_x(Q_r(z_0))};\\\nonumber
&&C_u(r,z_0)=r^{-2}\|u\|^3_{L^3_tL^3_x(Q_r(z_0))};\quad\quad\quad\quad
\tilde{C}_u(r,z_0)=r^{-2}\|u-(u)_r\|^3_{L^3_tL^3_x(Q_r(z_0))};\\\nonumber
&&C_{\sqrt{n}}(r,z_0)=r^{-2}\|{\sqrt{n}}\|^3_{L^3_tL^3_x(Q_r(z_0))};\quad\quad
C_{\nabla\sqrt{c}}(r,z_0)=r^{-2}\|{\nabla\sqrt{c}}\|^3_{L^3_tL^3_x(Q_r(z_0))};\\\nonumber
&&D(r,z_0)=r^{-2}\|P\|^{\frac32}_{L^{\frac32}_tL^{\frac32}_x(Q_r(z_0))}; \quad\quad\quad\quad
M(r,z_0) =r^{-1}\|n \ln n\|_{L^\infty_tL^1_x(Q_r(z_0))}; \\\nonumber
&&N(r,z_0) = r^{-2}\|n \ln n\|^{\frac32}_{L^{\frac32}_tL^{\frac32}_x(Q_r(z_0))}.
\een
For simplicity, we denote $Q_r(0)$ by $Q_r$, and  we will use the following notations:
$A_u(r,0)=A_u(r)$, $E_u(r,0)=E_u(r)$, etc. Moreover, let
\ben\label{A.E}\nonumber
A_{u,{\nabla\sqrt{c}},{\sqrt{n}}}(r)=A_u(r)+A_{\nabla\sqrt{c}}(r)+A_{\sqrt{n}}(r);\\ \nonumber
E_{u,{\nabla\sqrt{c}},{\sqrt{n}}}(r)=E_u(r)+E_{\nabla\sqrt{c}}(r)+E_{\sqrt{n}}(r);\\\nonumber
C_{u,\nabla\sqrt c,\sqrt n}(r)=C_u(r)+C_{\nabla\sqrt{c}}(r)+C_{\sqrt{n}}(r).
\een

Before proving Theorem \ref{thm:3}, we first prove the following proposition.
\begin{proposition}\label{lem:1torho0}
Under the assumptions of Theorem \ref{thm:3},  let $\rho_0 \in (0,1)$. If there exists a constant $ \varepsilon_3\leq  \frac{\varepsilon^2_1}{5(1+\|\chi\|_{0})^{20}(1+\|c_0\|_{L^\infty}+\|\nabla \phi\|_{L^\infty})^{40}}$ such that 
\ben\label{eq:condition-32'}
N(\rho_0,z_0)  + C_{\sqrt{n}}(\rho_0,z_0) + C_{\nabla\sqrt{c}}(\rho_0,z_0) + C_u(\rho_0,z_0)+D(\rho_0,z_0)\leq \varepsilon_3,
\een
for some $ \rho_0\leq(\varepsilon_3)^{2} $, 
then $z_0$ is a regular point.
\end{proposition}


\begin{proof} Without loss of generality, let $z_0 = (0,0)$, then (\ref{eq:scaling'}) and $\eqref{eq:condition-32'}$ imply that

\ben\label{4,1}
\int_{Q_1} |n_{\rho_0}|^\frac32 + |\nabla \sqrt{c_{\rho_0}}|^3 + |u_{\rho_0}|^3+|P_{\rho_0}|^\frac32 \leq \varepsilon_3.
\een
The remaining part is to estimate the term of $ \int_{Q_{1}} |n_r\ln n_r|$. Note that
\beno
&&\int_{Q_{1}} |n_{\rho_0}\ln n_{\rho_0}|^{\frac32}\\ \nonumber
&\leq &  {\rho_0}^{-2}\int_{Q_{\rho_0}} |n\ln ({\rho_0}^2n)|^{\frac32}\\ \nonumber
&\leq &  {\rho_0}^{-2}\int_{Q_{\rho_0}\cap\{n<{\rho_0}^{-\frac32}\}} |n\ln ({\rho_0}^2n)|^{\frac32}+ {\rho_0}^{-2}\int_{Q_{\rho_0}\cap\{{\rho_0}^{-\frac32}\leq n\leq {\rho_0}^{-2}\}} |n\ln ({\rho_0}^2n)|^{\frac32}\\
&&+ {\rho_0}^{-2}\int_{Q_{\rho_0}\cap\{n>{\rho_0}^{-2}\}} |n\ln ({\rho_0}^2n)|^{\frac32}:= M_1'+M_2'+M_3'.
\eeno
For $ M_{1}', $ note that $ C(\ln \varepsilon_3)^{\frac32}\varepsilon_3^{\frac12}\leq 1$ for a suitable $\varepsilon_1$, then by (\ref{eq:condition-32'}) we have 
\ben\label{4,2} \nonumber
M_1' &\leq&  {\rho_0}^{-2}\int_{Q_{\rho_0}\cap\{n<{\rho_0}^{-\frac32}\}}( |n\ln n|+2n|\ln {\rho_0}|)^{\frac32}dx \\
&\leq& \varepsilon_3+2^{\frac32}|\ln \rho_0|^{\frac32}\varepsilon^{\frac32}_3\leq 2\varepsilon_3.
\een
For $ M_{2}'$ and $ M_{3}',$ direct calculations indicate that
\ben\label{4,3} \nonumber
M_2' &\leq&  {\rho_0}^{-2}\int_{Q_{\rho_0}\cap\{{\rho_0}^{-\frac32}\leq n\leq {\rho_0}^{-2}\}} |n\ln ({\rho_0}^2n)|^{\frac32}\\ \nonumber
&\leq &  {\rho_0}^{-2}\int_{Q_{\rho_0}\cap\{{\rho_0}^{-\frac32}\leq n\leq {\rho_0}^{-2}\}} |n\ln ({\rho_0}^{-\frac12})|^{\frac32}\\
&\leq &  {\rho_0}^{-2}\int_{Q_{\rho_0}} |n\ln n|^{\frac32}\leq \varepsilon_3,
\een
and
\ben\label{4,4} \nonumber
M_3' &\leq&  {\rho_0}^{-2}\int_{Q_{\rho_0}\cap\{n> {\rho_0}^{-2}\}} |n\ln ({\rho_0}^2n)|^{\frac32}\\
&\leq &  {\rho_0}^{-2}\int_{Q_{\rho_0}\cap\{n> {\rho_0}^{-2}\}} |n\ln n|^{\frac32} dx\leq \varepsilon_3.
\een
Combining $\eqref{4,1}$, $\eqref{4,2}$, $\eqref{4,3}$ and $\eqref{4,4}$, we have 
\beno
&&\int_{Q_1} |n_{\rho_0}|^\frac32 + |n_{\rho_0} \ln n_{\rho_0}|^\frac32 + |\nabla \sqrt{c_{\rho_0}}|^3 + |u_{\rho_0}|^3+|P_{\rho_0}|^{\frac32} \\&&\leq 5\varepsilon_3\leq \frac{\varepsilon_1^2}{(1+\|\chi\|_{0})^{20}(1+\|(c_0)_{\rho_0}\|_{L^\infty}+\|\nabla \phi_{\rho_0}\|_{L^\infty})^{40}},
\eeno
where we use $\|\nabla \phi_{\rho_0}\|_{L^\infty}\leq\|\nabla \phi\|_{L^\infty} $ and $\| (c_0)_{\rho_0}\|_{L^\infty}\leq\|c_0\|_{L^\infty} $.
By (\ref{eq:condition-32}), we know that $(n_{\rho_0}, \nabla\sqrt{c_{\rho_0}}, u_{\rho_0})$ are regular at $(0,0)$. The proof of Proposition \ref{lem:1torho0} is complete.
\end{proof}

In the following, we want to prove Theorem \ref{thm:3}. The proof is divided into five steps.

{\bf Step I: Local energy estimate.}
Set
\begin{eqnarray*}
\psi = \left\{
    \begin{array}{llll}
    \displaystyle 1, \quad (x,t) \in Q_r,\\
    \displaystyle 0, \quad (x,t) \in Q_{2r}^c,\\
    \end{array}
 \right.
\end{eqnarray*}
in local energy inequality, recall the inequality (\ref{local energy inequality}), there holds
\ben\label{2,1} \nonumber
&&\int_{B_r}(n \ln n)(\cdot,t) + 4 \int_{Q_r} |\nabla \sqrt{n}|^2+\frac{2}{\Theta_0}  \int_{B_r} (|\nabla \sqrt{c}|^2)(\cdot,t)\\ \nonumber
&&+\frac{4}{3\Theta_0}\int_{Q_r} |\Delta \sqrt{c}|^2 +\frac{18}{\Theta_0}\|{c_0}\|_{L^{\infty}}\int_{B_r}(|u|^2)(\cdot,t) + \frac{18}{\Theta_0}\|{c_{0}}\|_{L^{\infty}}\int_{Q_r} |\nabla u|^2\\ \nonumber
&\leq& C (1 + \|\chi\|_0 \|c_0\|_{L^\infty}^\frac12 + \|\nabla \phi\|_{L^\infty}) \left(\|n\|_{L^\frac32(Q_{2r})}^\frac32 + \|n \ln n\|_{L^\frac32(Q_{2r})} + \|n \ln n\|_{L^\frac32(Q_{2r})}^\frac32  \right)\\ \nonumber
&&+ C (1 + \|\chi\|_0 \|c_0\|_{L^\infty}^\frac12 + \|\nabla \phi\|_{L^\infty}) \left(\|\nabla \sqrt c\|_{L^3(Q_{2r})}^3+ \|\nabla \sqrt c\|_{L^3(Q_{2r})}^2 \right)\\\nonumber&&
+ C (1 +  \|c_0\|_{L^\infty} + \|\nabla \phi\|_{L^\infty})^2 \left( \|u\|_{L^3(Q_{2r})}^2 + \|u\|_{L^3(Q_{2r})}^3 + \|P\|_{L^\frac32(Q_{2r})}^\frac32\right)
.
\een
That is
\ben\label{5, 1}
&&r^{-1} \sup_t \int_{B_r} n \ln n + A_{\nabla\sqrt c,u}(r) + E_{\sqrt n,\nabla\sqrt c,u}(r) \\\nonumber&\leq & C(\Theta_0 )(1 + \|\chi\|_0)\left(1+\|c_0\|_{L^\infty}+\|\nabla \phi\|_{L^\infty}\right)^2 \left(C_{\sqrt n,\nabla\sqrt c,u}(2r) + D(2r) + 1 + N(2r)\right).
\een
Multiplying $(\ref{eq:GKS})_1$ with $\psi$ and integrating by parts, we arrive
\beno
\int_{B_r} (n \psi)(\cdot,t) = \int_{Q_{2r}} n  (\partial_t \psi + \Delta \psi)
+ \int_{Q_{2r}} n  u \cdot \nabla \psi+ \int_{Q_{2r}} n \chi(c) \nabla c \cdot \nabla \psi,
\eeno
which means 
\ben\label{5, 2}
A_{\sqrt n}(r) \leq (1 + \|\chi\|_0)\left(1+\|c_0\|_{L^\infty}+\|\nabla \phi\|_{L^\infty}\right)\left(C_{\sqrt n,\nabla\sqrt c,u}(2r) + 1\right).
\een
Noting that
\ben\label{n lnn}\nonumber
\sup_{t\in(-r^2,0)} \int_{B_r} |n \ln n| &\leq& \sup_{t\in(-4r^2,0)} \int_{B_{2r}} |n \ln n| \psi \\\nonumber
&\leq& \sup_{t\in(-4r^2,0)}\left(\int_{B_{2r}} n \ln n \psi - 2\int_{B_{2r} \cap \{x;n < 1\}} n \ln n \psi \right)   \\\nonumber
&\leq& \sup_{t\in(-4r^2,0)} \int_{B_{2r}} n \ln n \psi + 4e^{-1} \sup_{t\in(-4r^2,0)} \int_{B_{2r} \cap \{x;n < 1\}} n^\frac12 \psi \\
&\leq& \sup_{t\in(-4r^2,0)} \int_{B_{2r}} n \ln n \psi + 4e^{-1} |B_{1}|(2r)^{3},
\een
 by $\eqref{5, 1}$, $\eqref{5, 2}$ and (\ref{n lnn}), we arrive
\ben\label{4E}
M(r) + A_{\sqrt n,\nabla\sqrt c,u}(r) + E_{\sqrt n,\nabla\sqrt c,u}(r) \leq C \left(C_{\sqrt n,\nabla\sqrt c,u}(2r) + D(2r) + 1 + N(2r)\right).
\een

{\bf Step II: Estimate of the non-scale quantity $N(r)$.} Let $0 < 4r \leq \rho < 1$.
Consider the following estimate:
\beno
N(r) &=& r^{-2} \int_{Q_r} |n \ln n|^\frac32 dx dt \\\
&\leq& r^{-2} \int_{Q_r} |n \ln n - (n \ln n)_\rho|^\frac32 dx dt + r^{-2} \int_{Q_r} |(n \ln n)_\rho|^\frac32 dx dt \\
&:=& N_1(r) + N_2(r),
\eeno
where $(n \ln n)_\rho = |B_\rho|^{-1} \int_{B_\rho} n \ln n dx$. Set $ I_{\rho}=(-\rho^{2},0), $ for the term of $N_2(r)$, by the definition of $(n \ln n)_\rho$, there holds
\beno
N_2(r) \leq r\int_{I_\rho}\rho^{-\frac92}\left(\int_{B_\rho}n\ln ndx\right)^{\frac32}dt\leq  r\rho^{-3}\int_{Q_\rho}|n \ln n|^{\frac32}dxdt\leq  \left(\frac r\rho\right) N(\rho).
\eeno
For the term of $N_1(r)$, since $W^{1,1}(\mathbb{R}^3) \hookrightarrow L^\frac32(\mathbb{R}^3)$, we have
\ben\label{5,5}\nonumber
N_1(r) &\leq& C r^{-2} \int_{I_\rho} \left(\int_{B_\rho} |\nabla (n \ln n)| dx\right)^\frac32 dt\\\nonumber &\leq& C r^{-2} \int_{I_\rho} \left(\int_{B_\rho} \left|\nabla n^\frac12 \left[n^\frac12 (1 + \ln n)\right]\right| dx\right)^\frac32 dt \\
&\leq& C r^{-2} \left(\int_{Q_\rho} |\nabla \sqrt n|^2 dx dt \right)^\frac34 \left(\int_{I_\rho} \left(\int_{B_\rho} \left|n^\frac12 (1 + \ln n)\right|^2 dx\right)^3 dt \right)^\frac14.
\een
If $0< n \leq 1$,  we have $n^\frac12 |1 + \ln n|\leq Cn^\frac38$, which is bounded and 
\ben\label{5, 4}
\left(\int_{I_\rho} \left(\int_{B_\rho} \left|n^\frac12 (1 + \ln n)\right|^2 dx\right)^3 dt \right)^\frac14\leq C \rho^\frac{17}{16}\|n\|^{\frac{9}{16}}_{L^1}.
\een
If $n > 1$, there holds 
\beno
n^\frac12 (1 + \ln n) \leq C n^{\frac12 + \gamma} \quad {\rm for ~~ any} \quad \gamma > 0.
\eeno
Then we have
\ben\label{N 1}
N_1(r) \leq C r^{-2} \left(\int_{Q_\rho} |\nabla \sqrt n|^2 dx dt \right)^\frac34 \left(\int_{I_\rho} \left(\int_{B_\rho} \left|n^\frac12\right|^{2+4\gamma} dx\right)^3 dt \right)^\frac14.
\een
For some $\gamma$ which satisfy
\beno
\frac2{3(2+4\gamma)} + \frac3{2+4\gamma} \geq \frac32,
\eeno
we have 
\ben\label{5, 3}
\left(\int_{I_\rho} \left(\int_{B_\rho} \left|n^\frac12\right|^{2+4\gamma} dx\right)^3 dt \right)^\frac14 \leq C \rho^{\frac54-3\gamma} (A_{\sqrt n}(\rho) + E_{\sqrt n}(\rho))^\frac{3(1+2\gamma)}4.
\een
$\eqref{5, 4}$, $(\ref{N 1})$ and $\eqref{5, 3}$ imply that
\beno
N_1(r) &\leq& C \left(\frac \rho r\right)^2 \left[ \rho^{-3\gamma} E^{\frac34}_{\sqrt n}(\rho) \left(A_{\sqrt n}(\rho) + E_{\sqrt n}(\rho)\right)^\frac{3+6\gamma}4\right]\\&&+~ C\rho^{\frac38}\left(\frac \rho r\right)^2A^{\frac9{16}}_{\sqrt n}(\rho) E^{\frac34}_{\sqrt n}(\rho) .
\eeno
Collecting $N_1(r)$ and $N_2(r)$, for any $ 0 < \gamma \leq \frac19$, there holds
\ben\label{4N}\nonumber
N(r) &\leq& C  \left(\frac \rho r\right)^2 \left[ \rho^{-3\gamma} E^{\frac34}_{\sqrt n}(\rho) \left(A_{\sqrt n}(\rho) + E_{\sqrt n}(\rho)\right)^\frac{3+6\gamma}4\right] + \left(\frac r\rho\right) N(\rho)\\
&&+~ C\rho^{\frac38}\left(\frac \rho r\right)^2A^{\frac9{16}}_{\sqrt n}(\rho) E^{\frac34}_{\sqrt n}(\rho).
\een

{\bf Step III: Estimate of the nonlinear terms $C_{\sqrt n,\nabla \sqrt c,u}(r)$.}
Consider the following estimate:
\beno
C_u(r) = r^{-2} \int_{Q_r} |u|^3 dx dt \leq r^{-2} \int_{Q_r} |u - u_\rho|^3 dx dt + r^{-2} \int_{Q_r} |u_\rho|^3 dx dt.
\eeno
By the definition of $u_\rho$, there holds
\beno
r^{-2} \int_{Q_r} |u_\rho|^3 dx dt &\leq& C r \int_{I_\rho} \rho^{-9} \left(\int_{B_\rho} u dx \right)^3 dt \\&\leq& C r \rho^{-3} \int_{Q_\rho} |u|^3 dx dt = \left(\frac r \rho\right) C_u(\rho).
\eeno
By embedding inequality, there holds
\beno
\|u - u_\rho\|_{L^3(B_\rho)} \leq \|u - u_\rho\|_{L^2(B_\rho)}^\frac12  \|u - u_\rho\|_{L^6(B_\rho)}^\frac12 \leq C \|u\|_{L^2(B_\rho)}^\frac12 \|\nabla u\|_{L^2(B_\rho)}^\frac12.
\eeno
Then
\beno
r^{-2} \int_{Q_r} |u - u_\rho|^3 dx dt &\leq & C r^{-2} \int_{I_\rho} \|u\|_{L^2(B_\rho)}^\frac32 \|\nabla u\|_{L^2(B_\rho)}^\frac32 dt \\
&\leq &C r^{-2} \rho^2 A^\frac34_u(\rho) E^\frac34_u(\rho).
\eeno
The estimates of $C_{\sqrt n}(r)$ and $C_{\nabla \sqrt c}(r)$ are similar, we omit them. Finally, we have 
\ben\label{4C}
C_{\sqrt n, \nabla \sqrt c, u}(r) \leq C \left(\frac r \rho\right) C_{\sqrt n, \nabla \sqrt c, u}(\rho) + C\left(\frac \rho r\right)^2 A^\frac34_{\sqrt n, \nabla \sqrt c, u}(\rho) E^\frac34_{\sqrt n, \nabla \sqrt c, u}(\rho).
\een

{\bf Step IV: Estimate of the pressure $D(r)$.} Let $\eta(x) \geq 0$ be supported in $B_\rho$ with $\eta = 1$ in $B_\frac\rho2$, and
\beno
P_1(x,t) &=& \int_{\mathbb{R}^3} \frac1{4\pi|x-y|} (\partial_i \partial_j ((u_i-(u_i)_\rho) (u_j-(u_j)_\rho) \eta)(y,t) dy\\
&&+ \int_{\mathbb{R}^3} \frac1{4\pi|x-y|} \left[\nabla \cdot ((n-n_\rho)\nabla\phi \eta) + \nabla \cdot (n_\rho\nabla\phi \eta))\right](y,t) dy.
\eeno
In the following, we estimate the term $\nabla \cdot (n_\rho\nabla\phi \eta)$ in detail. Integration by parts and direct calculations yield that
\ben\label{n_rho}
\int_{\mathbb{R}^3} \frac1{4\pi|x-y|} \nabla \cdot (n_\rho\nabla\phi \eta)(y,t) dy \leq \int_{\mathbb{R}^3} \frac1{|x-y|^2} |n_\rho \nabla\phi \eta|(y,t) dy.
\een
By H\"{o}lder inequality, we obtain
\ben\label{|x-y|}
\left\|\int_{\mathbb{R}^3} \frac1{|x-y|^2} |n_\rho \nabla\phi \eta|(y,t) dy\right\|_{L^\frac32(B_\rho)} \leq \rho^\frac12 \left\|\int_{\mathbb{R}^3} \frac1{|x-y|^2} |n_\rho \nabla\phi \eta|(y,t) dy\right\|_{L^2(B_\rho)}
\een
Using Lemma \ref{lem 3}, there holds
\ben\label{n nabla psi}
\rho^\frac12 \left\|\int_{\mathbb{R}^3} \frac1{|x-y|^2} |n_\rho \nabla\phi \eta|(y,t) dy\right\|_{L^2(B_\rho)}
\leq C \rho^\frac12\|n_\rho \nabla\phi \eta\|_{L^\frac65(\mathbb{R}^3)}.
\een
Moreover, let
\beno
P_2(x,t) =P(x,t) - P_1(x,t)
\eeno
which implies that
\beno
\Delta P_2 = 0 \quad {\rm in} \quad B_\frac\rho2.
\eeno
Let $0 < 4r < \rho \leq 1$, by Lemma \ref{lem 1} , we have
\ben\label{ine:p2} \nonumber
\int_{{B_r}}|P_2|^{\frac32}dx &\leq& C\left(\frac r\rho\right)^{3} \int_{B_{\frac34 \rho}}|P_2|^{\frac32}dx \\
&\leq& C \left(\frac r\rho\right)^{3} \int_{B_{\rho}}|P|^{\frac32}dx+C \left(\frac r\rho\right)^{3}\int_{B_{\rho}}|P_1|^{\frac32}dx.
\een
Besides, Calderon-Zygmund estimates and Riesz potential estimates yield that
\ben\label{ine:p1} \nonumber
\int_{B_{\rho}}|P_1|^{\frac32}dx &\leq& C \int_{B_{\rho}}|u-u_\rho|^3+C \rho^{\frac 34}\left(\int_{B_{\rho}}|(n-n_\rho)\nabla \phi|^\frac65dx\right)^{\frac54} \\
&&+ C \rho^\frac34 \left(\int_{B_\rho} |n_\rho \nabla \phi \eta|^\frac65 dx\right)^\frac54.
\een
Combining (\ref{ine:p1}) and (\ref{ine:p2}) and noting that $r<\rho$, we have
\beno
r^{-2} \int_{Q_r} |P|^\frac32 &\leq& r^{-2} \int_{Q_r} |P_1|^\frac32 + r^{-2} \int_{Q_r} |P_2|^\frac32 \\
&\leq& C\left(1+\left(\frac r\rho\right)^{3}\right)r^{-2} \int_{Q_\rho} |P_1|^\frac32+ C r^{-2} \left(\frac r\rho\right)^{3} \int_{Q_{\rho}}|P|^{\frac32}dxdt \\
&\leq& C r^{-2} \int_{Q_{\rho}}|u-u_\rho|^3 + C r^{-2} \int_{I_\rho}\rho^{\frac 34}\left(\int_{B_{\rho}}|(n-n_\rho)\nabla \phi|^\frac65dx\right)^{\frac54}dt \\
&&+ C r^{-2} \int_{I_r}\rho^\frac34 \left(\int_{B_\rho} |n_\rho \nabla \phi \eta|^\frac65 dx\right)^\frac54 dt + C r^{-2} \left(\frac r\rho\right)^{3} \int_{Q_{\rho}}|P|^{\frac32}dxdt.
\eeno
Using H\"{o}lder inequality, we have
\beno\label{pressure}
r^{-2} \int_{Q_r} |P|^\frac32
&\leq& C \left(\frac\rho r\right)^2 \rho^{-2} \int_{Q_{\rho}}|u-u_\rho|^3 + C\|\nabla\phi\|_{L^\infty_{x}}^\frac32 \left(\frac\rho r\right)^2 \rho^\frac32 \left(\rho^{-\frac53} \int_{Q_\rho} |n-n_\rho|^\frac53dxdt\right)^\frac9{10} \nonumber\\
&& + C r^{-2} \int_{I_r} \rho^\frac34 \left(\int_{B_\rho} |n_\rho \nabla \phi \eta|^\frac65 dx\right)^\frac54 dt + C \left(\frac r\rho\right) \rho^{-2} \int_{Q_{\rho}}|P|^{\frac32}dx,
\eeno
by (\ref{4C}), which means
\ben\label{4D} \nonumber
D(r) &\leq & C\left(\frac r\rho\right) D(\rho) + C\left(\frac\rho r\right)^2 A_u(\rho)^\frac34 E_u(\rho)^\frac34\\&&+~C (1+\|\nabla \phi\|_{L^\infty}^{\frac32})\left(\frac\rho r\right)^2 \rho^\frac32 A_{\sqrt n}(\rho)^\frac34 E_{\sqrt n}(\rho)^\frac3{4} \\\nonumber
&&+~C(1+\|\nabla \phi\|_{L^\infty}^{\frac32})\left(\frac\rho r\right)^2 \rho^\frac32 A_{\sqrt n}(\rho)^\frac32.
\een

{\bf Step V: Iteration argument.}
Let 
\beno
G(r) = N(r) + D(r) + C_{\sqrt n, \nabla \sqrt c,u}(r).
\eeno
By $\eqref{4C}$, $\eqref{4D}$ and $\eqref{4N}$, for any $0 < 4r \leq \rho$, there holds
\beno
G(r) &\leq& C\left(\frac r\rho\right) G\left(\frac \rho2\right) + C\left(\frac \rho r\right)^2 \left[ \rho^{-3\gamma} E^{\frac34}_{\sqrt n}\left(\frac \rho2\right) \left(A_{\sqrt n}\left(\frac \rho2\right) + E_{\sqrt n}\left(\frac \rho2\right)\right)^\frac{3+6\gamma}4\right] \\
&&+ C\left(\frac \rho r\right)^2 A^\frac34_{\sqrt n, \nabla \sqrt c, u}\left(\frac \rho2\right) E^\frac34_{\sqrt n, \nabla \sqrt c, u}\left(\frac \rho2\right) +C \left(\frac\rho r\right)^2 \rho^\frac32 A^\frac34_{\sqrt n}\left(\frac \rho2\right) E^\frac3{4}_{\sqrt n}\left(\frac \rho2\right) \\
&&+ C\left(\frac\rho r\right)^2 \rho^\frac32 A^\frac32_{\sqrt n}\left(\frac \rho2\right)+C\left(\frac\rho r\right)^2 A^{\frac9{16}}_{\sqrt n}(\rho)E^{\frac34}_{\sqrt n}(\rho)\\
&\leq& C\left(\frac r\rho\right) G\left(\frac \rho2\right) + C\left(\frac \rho r\right)^2 \left(A_{\sqrt n, \nabla \sqrt c, u}\left(\frac \rho2\right) + E_{\sqrt n, \nabla \sqrt c, u}\left(\frac \rho2\right)\right) \\
&&\times (1+\|\nabla \phi\|_{L^\infty}^{\frac32}) \left[\rho^{-3\gamma} E^\frac{1+3\gamma}2 _{\sqrt n}\left(\frac \rho2\right)+ E^\frac12_{\sqrt n,\nabla \sqrt c,u}\left(\frac \rho2\right) + E^{\frac{5}{16}}_{\sqrt n}(\rho)+\rho^\frac32 A^\frac12_{\sqrt n}\left(\frac \rho2\right)\right]
.
\eeno
By $\eqref{4E}$, there holds
\beno
A_{\sqrt n, \nabla \sqrt c, u}\left(\frac \rho2\right) + E_{\sqrt n, \nabla \sqrt c, u}\left(\frac \rho2\right) \leq C\left(G(\rho) + 1\right),
\eeno
which means
\beno
&&G(r) \leq C(1+\|\nabla \phi\|_{L^\infty}^{\frac32}) G(\rho)\nonumber\\
&&\times \left\{\frac r\rho + \left(\frac \rho r\right)^2 \left[\rho^{-3\gamma} E^\frac{1+3\gamma}2_{\sqrt n}(\rho) + E^\frac12_{\sqrt n,\nabla \sqrt c,u}(\rho)+E^{\frac5{16}}_{\sqrt n}(\rho) + \rho^\frac32 A^\frac12_{\sqrt n}(\rho)\right] \right\}  \\
&&+ C(1+\|\nabla \phi\|_{L^\infty}^{\frac32})\left(\frac \rho r\right)^2 \left[\rho^{-3\gamma} E^\frac{1+3\gamma}2_{\sqrt n}(\rho) + E^\frac12_{\sqrt n,\nabla \sqrt c,u}(\rho)+E^{\frac5{16}}_{\sqrt n}(\rho)  + \rho^\frac32 A^\frac12_{\sqrt n}(\rho)\right] .
\eeno
Choosing $r = \theta \rho$ with $\theta \in (0,\frac12)$, there holds
\ben\label{ine:G} \nonumber
&&G(\theta \rho)\nonumber\\
 &\leq & C(1+\|\nabla \phi\|_{L^\infty}^{\frac32}) \left\{\theta + \theta^{-2}\left[\rho^{-3\gamma} E^\frac{1+3\gamma}2_{\sqrt n}(\rho)+E^{\frac5{16}}_{\sqrt n}(\rho) + E^\frac12_{\sqrt n,\nabla \sqrt c,u}(\rho) + \rho^\frac32 A^\frac12_{\sqrt n}(\rho)\right]\right\} G(\rho) \nonumber\\ \nonumber
&&+ C(1+\|\nabla \phi\|_{L^\infty}^{\frac32})\theta^{-2} \left[\rho^{-3\gamma} E^\frac{1+3\gamma}2_{\sqrt n}(\rho)+E^{\frac5{16}}_{\sqrt n}(\rho) + E^\frac12_{\sqrt n,\nabla \sqrt c,u}(\rho) + \rho^\frac32 A^\frac12_{\sqrt n}(\rho)
 \right] . \\
\een
Since $\|n\|_{L^\infty_t L^1_x} \leq \|n_0\|_{L^\infty_t L^1_x}$, we have
\ben\label{ine:1}
\rho^\frac32 A^\frac12_{\sqrt n}(\rho) = \rho \left(\sup_{t \in (-\rho^2,0)}\int_{B_\rho} n dx \right)^\frac12 \leq A_0 \rho,
\een
where $A_0=\|n_0\|_{L^\infty_t L^1_x}$ is a constant,  which depends on the norm of initial value. Let
\ben\label{eq:epsilon'}
\varepsilon=\frac{\varepsilon^8_1}{625(1+\|\chi\|_{0})^{80}(1+\|\nabla\phi\|_{L^\infty}+\|c_0\|_{L^\infty})^{160}}.\een
By $\eqref{assume:3}$, for any $\rho \in (0,1)$ without loss of generality, there holds 

\ben\label{ine:2}
E_{\sqrt n}(\rho) \leq \varepsilon \rho^{\delta_0}, \quad E_{\nabla \sqrt c,u}(\rho) \leq \varepsilon.
\een
Putting $\eqref{ine:1}$ and $\eqref{ine:2}$ into $\eqref{ine:G}$, we arrive
\beno
G(\theta \rho) &\leq& C \left[\theta + \theta^{-2} \left(\rho^{\frac12\delta_0 + (\frac32\delta_0-3)\gamma} \varepsilon^\frac{1+3\gamma}2 + \varepsilon^\frac5{16} + A_0 \rho\right)\right] G(\rho)(1+\|\nabla \phi\|_{L^\infty}^{\frac32}) \\
&&+ C \theta^{-2} \left(\rho^{\frac12\delta_0 + (\frac32\delta_0-3)\gamma} \varepsilon^\frac{1+3\gamma}2 + \varepsilon^\frac5{16} + A_0 \rho\right)(1+\|\nabla \phi\|_{L^\infty}^{\frac32}).
\eeno
Here the constant $C$ is independent on $\rho$, $\varepsilon$, $\theta$ and $A_0$.
If there exists $\gamma \in (0,\frac19]$ such that 
\ben\label{gammadelta0}
\frac12\delta_0 + (\frac32\delta_0-3)\gamma > 0,
\een
we have 
\beno
G(\theta \rho) &\leq& C(1+\|\nabla \phi\|_{L^\infty}^{\frac32}) \left[\theta + \theta^{-2} \left(\varepsilon^\frac5{16} + A_0 \rho\right)\right] G(\rho) \nonumber\\
&&+ C(1+\|\nabla \phi\|_{L^\infty}^{\frac32}) \theta^{-2} \left(\varepsilon^\frac5{16} + A_0 \rho\right).
\eeno
Obviously, there exists $\gamma \in (0,\frac19]$ satisfies $\eqref{gammadelta0}$ since for any $\delta_0 > 0$, there exists a small constant such that  
\beno
\gamma < \frac{\delta_0}{6-3\delta_0}.
\eeno
Now, choose $\theta = \theta_0 \in (0,\frac14)$ small enough such that
\beno
C\theta_0 \leq \frac 1{4(1+\|\nabla \phi\|_{L^\infty}^{\frac32})}.
\eeno
Then for 
\beno
\rho \leq \rho_0=\frac{\theta_0^2 \varepsilon}{8A_0 C(1+\|\nabla \phi\|_{L^\infty}^{\frac32})},
\eeno
we have 
\beno
G(\theta_0\rho) &\leq&\left(\frac14 + C(1+\|\nabla \phi\|_{L^\infty}^{\frac32}) \theta_0^{-2} \varepsilon^\frac5{16} + \frac18 \varepsilon\right) G(\rho) \nonumber\\
&&+ C (1+\|\nabla \phi\|_{L^\infty}^{\frac32})\theta_0^{-2} \varepsilon^\frac5{16}  + \frac18 \varepsilon.
\eeno
The constant $\varepsilon$ satisfy
\beno
\varepsilon^\frac1{16}\leq \frac{\theta_0^2}{8 C(1+\|\nabla \phi\|_{L^\infty}^{\frac32})}\leq \frac{1}{128 C^3(1+\|\nabla \phi\|_{L^\infty}^{\frac32})^3},
\eeno
due to (\ref{eq:epsilon'}).
We have 
\beno
G(\theta_0\rho_0) \leq \frac12 G(\rho_0) + 2 \varepsilon^\frac14.
\eeno
Noting that
\beno
G(\rho_0) = N(\rho_0) + C_{\sqrt n,\nabla \sqrt c,u}(\rho_0) + D(\rho_0) < \infty,
\eeno
by the classical iterative argument, we obtain
\beno
G(\theta_0^k \rho_0) \leq \frac 1{2^k} G(\rho_0) + 4\varepsilon^\frac 14.
\eeno
Let $k = k_0 = \left[\left(\ln 2\right)^{-1} \ln \left(G(\rho_0)\varepsilon^{-\frac14}\right)\right] + 1$, and by (\ref{eq:epsilon'}) we have
\beno
G(\theta_0^{k_0} \rho_0) \leq 5\varepsilon^\frac 14\leq \frac{\varepsilon^2_1}{(1+\|\chi\|_{0})^{20}(1+\|\nabla\phi\|_{L^\infty}+\|c_0\|_{L^\infty})^{40}}.
\eeno

By Proposition \ref{lem:1torho0}, we finish the proof.

\section{Proof of Corollary \ref{thm:4}}

\begin{definition}
For a set $E\subset \mathbb{R}^{n+1}$ and $\alpha\geq 0$, $Q_r(z_0)=B_r(x_0)\times(t_0-r^2,t_0)$ for $z_0=(x_0,t_0)$. Denote by $\mathcal{P}^\alpha(E)$ its $\alpha-$dimensional parabolic Hausdorff measure, namely,
\beno
\mathcal{P}^\alpha(E)=\liminf_{\delta\rightarrow0^{+}} \left\{\sum_{j=1}^{\infty}r_j^{\alpha}: E\subset\bigcup_jQ(z_j,r_j),r_j\leq \delta\right\}.
\eeno
\end{definition}

We will use a parabolic version of the Vitali covering lemma:
Let $\{J={Q_{z_\alpha,r_\alpha}}\}_\alpha$ be any collection of parabolic cylinders contained in a bounded subset of $\mathbb{R}^{4}$, and noting $J=J_x\times J_t$, there exist disjoint $Q_{z_j,r_j}\in J, j\in N$, such that any cylinder in $J$ is contained in $Q_{z_j,5r_j}$ for some $j$.

Letting
\beno
Q^{*}((x,t),r)=B(x,r)\times(t-\frac78r^2,t+\frac18r^2),
\eeno
it is a translation in time of $Q((x,t),r)$. Besides, $Q(z,\frac r2)\in Q^{*}(z,r)$.
Let $\mathcal{S}_R =\mathcal{S} \cap R$ for any compact set $R \subset Q_\frac12$. Fix any $\delta > 0$. 
Assume that for any $z_j = (x_j,t_j) \in \mathcal{S}_R$, by Theorem \ref{thm:3}, there exists $0 < r_{z_j}=r_j < \frac\delta{10}$ such that 
\beno
\int_{Q_{r_j}(z_0)} |\nabla \sqrt n|^2 + |\nabla^2 \sqrt c|^2 + |\nabla u|^2 \geq \frac12 r_j^{1+\delta_0} \varepsilon_4.
\eeno
Thus,
\beno
\mathcal{S}_R \subset \bigcup_{j\in N} Q^{\ast}(z_j,2r_{z_j}).
\eeno
Let $r_j = r_{z_j}$ and $\{Q_{r_j}(z_j)\}_{j\in \mathbb{N}}$ be the countable disjoint subcover guaranteed by the Vitali covering lemma, then
\beno
\mathcal{S}_R \subset \bigcup_{j \in \mathbb{N}} Q^{\ast}(z_j, {10r_j})\quad 10r_j<\delta.
\eeno
Note that $Q({z_j, \frac {r_j}2})\in Q^{*}({z_j,{r_j}})$ is disjoint, then
\beno
\sum 10r_j^{1+\delta_0} &\leq& \sum_j \frac{20}{\varepsilon_4} \left(\int_{Q_{r_j}(z_0)} |\nabla \sqrt n|^2 + |\nabla^2 \sqrt c|^2 + |\nabla u|^2\right).
\eeno
Since the bounded-ness of the right part, we have
\beno
\sum 10r_j^{1+\delta_0} < + \infty,
\eeno
which means that $S_R$ has Labesgue measure $0$. Since the finite covering theorem, we know that any open neighborhood $J = J_x \times J_t \subset {Q_1}$ of $S_R$ satisfies $Q_{r_z}(z) \subset J$,
\beno
\sum 5r_j^{1+\delta_0} \leq \sum_j \frac{C}{\varepsilon_0} \left(\int_{Q_{r_j}(z_0)} |\nabla \sqrt n|^2 + |\nabla^2 \sqrt c|^2 + |\nabla u|^2\right).
\eeno
By the arbitrarily of $J$, we can choose $J$ with arbitrarily small Lebesgue measure, therefore, the right side is arbitrarily small. Since $\delta > 0$ is arbitrary, we have
\beno
{\mathcal{P}}^{1+\delta_0}(\mathcal{S}_R) = 0.
\eeno
By the arbitrarily of $R$, we have
\beno
{\mathcal{P}}^{1+\delta_0}(\mathcal{S}) = 0.
\eeno

\section{Proof of Theorem \ref{thm:1}}
In this section, we follow the same route as in  \cite{CLW 2022}. The main difficulties lie in the estimates of $\int n\ln n$ and $\int|\Delta c|^2$.
Firstly, we present the following elementary lemma as a preparation for some estimates.
\begin{lemma}\label{heat kernel}
Set
\beno
\Psi_{n}(x,t) = \frac{1}{(r_{n}^2 - t )^\frac 32} \exp(- \frac{|x |^2}{4(r_{n}^2 - t )}),
\eeno
where $(x,t) \in \mathbb{R}^3 \times (-\infty,r_{n}^2)$. Letting $\xi(x,t)$ in $ Q_{r_3}$ be  a suitable cut-off function, which satisfies
\begin{eqnarray*}
 \xi(x,t)=\left\{
    \begin{array}{llll}
    \displaystyle 1,\quad {\rm in}\quad Q_{r_4}, \\
    \displaystyle 0,\quad {\rm in}\quad Q_{r_3}^{c},\\
    \end{array}
 \right.
\end{eqnarray*}
the properties of $\phi_n = \Psi_n \xi$ are as follows:
\begin{itemize}
\item [i)]
 $C^{-1} r_{n}^{-3} \leq \phi_n(x,t) \leq C r_{n}^{-3}$ on $Q_{r_{n}}$ for $n\geq 2$;\\
\item [ii)] $\phi_n(x,t) \leq C r_k^{-3}$  for  $(x,t) \in Q_{r_k} \setminus Q_{r_{k+1}}$, \quad $1<k\leq n$;\\
\item [iii)]$|\nabla \phi_n(x,t) |\leq Cr_n^{-4}$ in $Q_{r_n}$,\quad $n\geq 2$;\\
\item [iv)] $|\nabla \phi_n(x,t)| \leq C r_k^{-4}$ on $Q_{r_{k-1}} \setminus Q_{r_{k}}$, \quad $1<k\leq n$;\\
\item [v)]$\left|\partial_t \phi + \Delta \phi\right| \leq C$ on $Q_{r_3}$;\\
\item [vi)] $\partial_t \phi + \Delta \phi=0$ on $Q_{r_4}$,
\end{itemize}
where $C$ is an absolute constant.
\end{lemma}
In this section, we would like to prove Theorem \ref{thm:1} by iterative argument and mathematical induction. It is sufficient to prove Proposition \ref{lem:induction}.

\begin{proposition}\label{lem:induction}
Assume that $(n,c,u)$ is a suitable weak solutions of (\ref{eq:GKS}) in $\mathbb{R}^3\times(-1,0)$ and $r_k = 2^{-k}$. Under the condition $\eqref{eq:condition}$, for any integer $k \geq 1$, there holds
\ben\label{ine:induction}
&&r_k^{-3}\sup_{-r_k^2<t<0} \int_{B_{r_k}} n  + |n \ln n| + |\nabla \sqrt{c}|^2 + |u|^2 \nonumber\\
&&+ r_k^{-3} \int_{Q_{r_k}} |\nabla \sqrt{n}|^2 + |\nabla^2 \sqrt{c}|^2 + |\nabla u|^2 + r_k^{-4} \int_{Q_{r_k}} |P-\bar P|^\frac32 \leq C_1 \varepsilon_0^\frac12,
\een
where $C_1>1$ is an absolute constant.
\end{proposition}

\begin{proof}

Obviously, (\ref{eq:condition}) implies that (\ref{ine:induction}) holds for $k = 1$. Assume that (\ref{ine:induction}) holds for the case of  $k=2,\cdots, N$. Next we would like to prove (\ref{ine:induction}) comes true when $k=N+1$.

{\bf Step I: Estimates from the local energy inequality.}

 Taking $\psi=\phi_{N+1}$ as a test function in the local energy inequality (\ref{local energy inequality}) and taking $\Omega=B_{r_3}$, we can write it as follows:

{\footnotesize

\beno
&&\int_{B_{r_3}} (n \ln n ~ \psi)(\cdot,t) +  C^{-1}r_{N+1}^{-3} \int_{Q_{r_3}} |\nabla \sqrt{n}|^2 +\frac{1}{C\Theta_0} r_{N+1}^{-3} \int_{B_{r_3}} (|\nabla \sqrt{c}|^2 )(\cdot,t)\\
&&+\frac{1}{C\Theta_0} r_{N+1}^{-3}\int_{Q_{r_3}} |\Delta \sqrt{c}|^2 +\frac{1}{C\Theta_0}\|{c_{0}}\|_{L^{\infty}} r_{N+1}^{-3} \int_{B_{r_3}}(|u|^2)(\cdot,t) \\
&&+ \frac{1}{C\Theta_0}\|{c_{0}}\|_{L^{\infty}} r_{N+1}^{-3} \int_{Q_{r_3}} |\nabla u|^2 +\frac{1}{C\Theta_0} r_{N+1}^{-3} \int_{Q_{r_3}} (\sqrt{c})^{-2} |\nabla \sqrt{c}|^4\\
&\leq&\int_{Q_{r_3}} n \ln n ~ (\partial_t \psi + \Delta \psi) + \int_{Q_{r_3}} n \ln n ~ u \cdot \nabla \psi +\int_{Q_{r_3}}n\chi(c)\nabla c\cdot \nabla \psi \\
&&+\int_{Q_{r_3}}n\ln n ~ \chi(c)\nabla c\cdot\nabla\psi + \frac{2}{\Theta_0}\int_{Q_{r_3}} |\nabla \sqrt{c}|^2 (\partial_t \psi + \Delta \psi) + \frac{2}{\Theta_0} \int_{Q_{r_3}} |\nabla \sqrt{c}|^2 u \cdot \nabla \psi\\
&&+\frac{18}{\Theta_0}||c_{0}||_{L^{\infty}}\int_{Q_{r_3}} |u|^2 \left(\partial_t \psi +\Delta \psi\right) +\frac{18}{\Theta_0}||c_{0}||_{L^{\infty}}\int_{Q_{r_3}} |u|^2 u\cdot\nabla\psi\\
&&+\frac{36}{\Theta_0}||c_{0}||_{L^{\infty}}\int_{Q_{r_3}} (P - \bar{P}) u \cdot \nabla \psi - \frac{36}{\Theta_0}||c_{0}||_{L^{\infty}}\int_{Q_{r_3}} n\nabla\phi \cdot u \psi \\
&:=& I_1 + I_2 + \cdots + I_{10},
\eeno
}
where $\bar{P}$ denotes the mean value of $P$ in  $B_{r_3}$.

Next, we estimate $I_1$ to $I_{10}$ term by term.

{\bf \underline{Estimate of $I_1$}:}
By $\eqref{ine:induction}$ and Lemma \ref{heat kernel}, we have
\beno
I_1 \leq C \int_{Q_{r_3}} |n \ln n| \leq C \varepsilon_0.
\eeno

{\bf \underline{Estimate of $I_2$}:}
In order to estimate $I_2$, we need the estimate of $n \ln n$. By embedding inequality
\beno
\|f\|_{L^q_tL^p_x}^2\leq C\left(\|f\|^2_{L^\infty_tL^2_x}+\|\nabla f\|^2_{L^2_tL^2_x}\right),\quad{\rm for}\quad\frac2q+\frac3p=\frac32,~~2\leq p\leq 6,
\eeno
and (\ref{ine:induction}) we have
\ben\label{ine:ncu}
\|n\|_{L_{t,x}^\frac53(Q_{r_k})} + \|\nabla \sqrt c\|_{L_{t,x}^\frac{10}3(Q_{r_k})}^2 + \|u\|_{L_{t,x}^\frac{10}3(Q_{r_k})}^2 \leq C C_1 r_k^3 \varepsilon_0^\frac12.
\een
Note that
\ben\label{n_0 1}
\lim_{n \rightarrow 0} n^\frac16 \ln n = 0,
\een
and
\ben\label{n_1 1}
\lim_{n \rightarrow \infty} n^{-\frac16}|\ln n|^{\frac32} = 0.
\een
Decompose $Q_{r_{k}}$ into $Q_{r_{k}} \cap \{n \leq 100\}$ and $Q_{r_{k}} \cap \{n>100\}$, and by (\ref{n_0 1}) and (\ref{n_1 1}), we arrive at
\beno\label{ine:estimate n ln n}
\int_{Q_{r_{k}}} |n \ln n|^{\frac32} \leq C \int_{Q_{r_{k}} \cap \{ n(x) \leq 100\}} |n|^{\frac43} + C \int_{Q_{r_{k}} \cap \{ n(x) > 100\}} |n|^\frac53,
\eeno
which is controlled by
\begin{equation}\label{ine:n ln n}
	\begin{aligned}
	\int_{Q_{r_{k}}} |n \ln n|^{\frac32}&\leq C \int_{Q_{r_{k}}} |n|^{\frac43} + r_{k}^5C C_1^\frac53 \varepsilon_0^\frac56 \\
	&\leq C r_{k}^5 C_1^{\frac43} \varepsilon_0^{\frac23} + C r_{k}^5 C_1^\frac53 \varepsilon_0^\frac56\leq  C r_{k}^5 C_1^\frac53 \varepsilon_0^\frac23,
	\end{aligned}
\end{equation}
where $C$ is an absolute constant and $C_1$ is from (\ref{ine:induction}). 
By $\eqref{ine:n ln n}$, $\eqref{ine:ncu}$, H\"{o}lder's inequality and Lemma \ref{heat kernel} for the property of $\psi$, we have
\beno
I_2 &\leq & C\sum_{k = 1}^N \int_{Q_{r_k} \setminus Q_{r_{k+1}}} \left|n \ln n ~ u \cdot \nabla \psi\right| + C\int_{Q_{r_{N+1}}} \left|n \ln n ~ u \cdot \nabla \psi\right|\\
&\leq& \sum_{k=1}^N C r_k^{-4} \|n \ln n\|_{L^{\frac32}(Q_{r_k})} \|u\|_{L^\frac{10}3(Q_{r_k})} r_k^{\frac16} + C r_{N+1}^{-4} \|n \ln n\|_{L^{\frac32}(Q_{r_N})} \|u\|_{L^\frac{10}3(Q_{r_N})} r_{N+1}^{\frac16} \\
&\leq& C C_1^\frac{29}{18} \sum_{k=1}^N r_k^{-4} r_k^{\frac16} r_k^\frac{10}{3} \varepsilon_0^\frac{4}{9} r_k^{\frac32} \varepsilon_0^\frac14 +C C_1^\frac{29}{18} r_{N+1}^{-4} r_{N+1}^{\frac16} r_N^\frac{10}{3} r_N^\frac32 \varepsilon_0^\frac94 \varepsilon_0^\frac14 \\
&\leq& C C_1^\frac{29}{18} \varepsilon_0^\frac{25}{36}.
\eeno

{\bf \underline{Estimate of $I_3$}. }
Using $\eqref{ine:n ln n}$, $\eqref{ine:ncu}$, H\"{o}lder's inequality and Lemma \ref{heat kernel} for the property of $\psi$, we get
\beno
I_3 &\leq & C \|\chi\|_0 \sum_{k=1}^N \int_{Q_{r_k} \setminus Q_{r_{k+1}}} \left|n \nabla c \cdot \nabla \psi\right| + C \|\chi\|_0 \int_{Q_{r_{N+1}}}\left| n \nabla c \cdot \nabla \psi \right| \\
&\leq& C \|\chi\|_0 \|c_0\|_{L^\infty}^\frac12 \sum_{k=1}^N r_k^{-4} \|n\|_{L^\frac53(Q_{r_k})} \|\nabla \sqrt{c}\|_{L^\frac{10}3(Q_{r_k})} r_k^\frac12 \\
&&+ C \|\chi\|_0 \|c_0\|_{L^\infty}^\frac12 r_{N+1}^{-4} \|n\|_{L^\frac53(Q_{r_N})} \|\nabla \sqrt{c}\|_{L^\frac{10}3(Q_{r_N})} r_{N+1}^\frac12 \\
&\leq& C \|\chi\|_0 \|c_0\|_{L^\infty}^\frac12 C_1^\frac32 \sum_{k=1}^N r_k^{-4} r_k^3 \varepsilon_0^\frac12 r_k^\frac32 \varepsilon_0^\frac14 r_k^\frac12 + C \|\chi\|_0 \|c_0\|_{L^\infty}^\frac12 C_1^\frac32 r_{N+1}^{-4} r_N^3 \varepsilon_0^\frac12 r_N^\frac32 \varepsilon_0^\frac14 r_k^\frac12 \\
&\leq& C \|\chi\|_0 \|c_0\|_{L^\infty}^\frac12 C_1^\frac32 \varepsilon_0^\frac34.
\eeno

{\bf \underline{Estimate of $I_4$}. } Since $u$ is similar as $\nabla \sqrt{c}$, using $\eqref{ine:n ln n}$, $\eqref{ine:ncu}$, H\"{o}lder's inequality and Lemma \ref{heat kernel} for the property of $\psi$ agian, we acquire
\beno
I_4 \leq C \|\chi\|_0 \|c_0\|_{L^\infty}^\frac12 C_1^\frac{29}{18} \varepsilon_0^\frac{25}{36}.
\eeno

{\bf \underline{Estimate of $I_5$ and $ I_{7} $}.}
$\eqref{ine:induction}$ and Lemma \ref{heat kernel} for the property of $\psi$ yield that
\beno
I_5 \leq \frac C {\Theta_0} \int_{Q_{r_3}} |\nabla \sqrt c|^2 \leq\frac C {\Theta_0} \varepsilon_0.
\eeno
Similarly, we get the estimate of $I_7$ as follows:
\beno
I_7 \leq  \frac{C}{\Theta_0} \|c_0\|_{L^\infty} \int_{Q_{r_3}} |u|^2 \leq  \frac{C}{\Theta_0} \|c_0\|_{L^\infty} \varepsilon_0.
\eeno

{\bf \underline{Estimate of $I_6$, $ I_{8}$ and $ I_{10} $}.}
Similar as the estimate of $I_2$, by $\eqref{ine:n ln n}$, $\eqref{ine:ncu}$, embedding inequality, H\"{o}lder's inequality and Lemma \ref{heat kernel} for the property of $\psi$, we obtain
\beno
I_6 &\leq &  \frac{C}{\Theta_0} \sum_{k=1}^N \int_{Q_{r_k} \setminus Q_{r_{k+1}}} |\nabla \sqrt{c}|^2 \left|u \cdot \nabla \psi\right| + \frac{C}{\Theta_0}\int_{Q_{r_{N+1}}} |\nabla \sqrt{c}|^2 \left|u \cdot \nabla \psi \right|\\
&\leq&  \frac{C}{\Theta_0} \sum_{k=1}^N r_k^{-4} \|\nabla \sqrt{c}\|_{L^\frac{10}3(Q_{r_k})}^2 \|u\|_{L^\frac{10}3(Q_{r_k})} r_k^\frac12 +  \frac{C}{\Theta_0} r_{N+1}^{-4} \|\nabla \sqrt{c}\|_{L^\frac{10}3(Q_{r_N})}^2 \|u\|_{L^\frac{10}3(Q_{r_N})} r_{N+1}^\frac12 \\
&\leq&  \frac{C}{\Theta_0} C_1^\frac32 \sum_{k=1}^N r_k^{-4} r_k^3 \varepsilon_0^\frac12 r_k^\frac32 \varepsilon_0^\frac14 r_k^\frac12 +  \frac{C}{\Theta_0} C_1^\frac32 r_{N+1}^{-4} r_N^3 \varepsilon_0^\frac12 r_N^\frac32 \varepsilon_0^\frac14 r_{N+1}^\frac12 \\
&\leq& \frac{C}{\Theta_0} C_1^\frac32 \varepsilon_0^\frac34.
\eeno
In the same way, we get the estimates of $I_8$ and $I_{10}$ as follows:
\beno
I_8 \leq C \frac{1}{\Theta_0} \|c_0\|_{L^\infty} C_1^\frac32 \varepsilon_0^\frac34,
\eeno 
and
\beno
I_{10} \leq C \frac{1}{\Theta_0} \|c_0\|_{L^\infty} \|\nabla \phi\|_{L^\infty} C_1^\frac32 \varepsilon_0^\frac34.
\eeno

{\bf \underline{Estimate of $I_9$}.}
Using (\ref{ine:induction}) and $\eqref{ine:ncu}$, we have
\beno
I_9 &\leq& \frac{18}{\Theta_0} \|c_0\|_{L^\infty} \sum_{k=1}^N r_k^{-4} \int_{Q_{r_k} \setminus Q_{r_{k+1}}} |P-\bar P| |u| + \frac{18}{\Theta_0} \|c_0\|_{L^\infty} r_{N+1}^{-4}\int_{Q_{r_{N+1}}} |P-\bar P| |u| \\
&\leq&  \frac{C}{\Theta_0} \|c_0\|_{L^\infty} \sum_{k=1}^N r_k^{-4}\left(\int_{Q_{r_k}} |P-\bar P|^\frac32 \right)^\frac23 \left(\int_{Q_{r_k}} |u|^\frac{10}{3} \right)^\frac{3}{10}r_k^\frac16 \\
&&+~ \frac{C}{\Theta_0} \|c_0\|_{L^\infty} r^{-4}_{N+1} \left(\int_{Q_{r_N}} |P-\bar P|^\frac32 \right)^\frac23\left(\int_{Q_{r_N}} |u|^\frac{10}{3}\right)^\frac{3}{10}r_N^\frac16 \\
&\leq&  \frac{C}{\Theta_0} \|c_0\|_{L^\infty} \sum_{k=1}^N r_k^{-4} r_k^\frac53 \left(C_1 \varepsilon_0^\frac12\right)^\frac12 r_k^\frac83 \left(C_1 \varepsilon_0^\frac12\right)^\frac23 \\&&+~ \frac{C}{\Theta_0} \|c_0\|_{L^\infty} r_{N+1}^{-4} r_N^\frac53 \left(C_1 \varepsilon_0^\frac12\right)^\frac12 r_N^\frac83 \left(C_1 \varepsilon_0^\frac12\right)^\frac23 \\
&\leq& \frac{C}{\Theta_0} \|c_0\|_{L^\infty} C_1^\frac76 \varepsilon_0^\frac7{12}.
\eeno

Collecting $I_1$ to $I_{10}$, there holds
\ben\label{Step1 con} \nonumber
&&\int_{B_{r_3}} (n \ln n ~ \psi)(\cdot,t) +  C^{-1}r_{N+1}^{-3} \int_{Q_{r_3}} |\nabla \sqrt{n}|^2 +\frac{1}{C\Theta_0} r_{N+1}^{-3} \int_{B_{r_3}} (|\nabla \sqrt{c}|^2 )(\cdot,t)\\\nonumber
&&+\frac{1}{C\Theta_0} r_{N+1}^{-3}\int_{Q_{r_3}} |\Delta \sqrt{c}|^2 +\frac{1}{C\Theta_0}\|{c_{0}}\|_{L^{\infty}} r_{N+1}^{-3} \int_{B_{r_3}}(|u|^2)(\cdot,t) \\\nonumber
&&+ \frac{1}{C\Theta_0}\|{c_{0}}\|_{L^{\infty}} r_{N+1}^{-3} \int_{Q_{r_3}} |\nabla u|^2 +\frac{1}{C\Theta_0} r_{N+1}^{-3} \int_{Q_{r_3}} (\sqrt{c})^{-2} |\nabla \sqrt{c}|^4\\\nonumber
&\leq& C \left(1 + \|\chi\|_0 \|c_0\|_{L^\infty}^\frac12 + \Theta_0^{-1} + \Theta_0^{-1} \|c_0\|_{L^\infty} (1+ \|\nabla \phi\|_{L^\infty})\right)\\\nonumber&&\times \left(\varepsilon_0 + C_1^\frac32 \varepsilon_0^\frac34 + C_1^\frac{29}{18} \varepsilon_0^\frac{25}{36} + C_1^\frac76 \varepsilon_0^\frac7{12}\right)\\
&\leq& C( \Theta_0)(1+\|\chi\|_{0})(1+ \|\nabla \phi\|_{L^\infty}+ \|c_0\|_{L^\infty})^2C_1^{\frac{29}{18}}\varepsilon^{\frac7{12}}_0.
\een



{\bf Step 2: Estimate of $r_{N+1}^{-4} \int_{Q_{r_{N+1}}} |P-\bar P|^\frac32$.}

For $0<2r<\rho\leq 1$, let $\eta \geq 0$ be supported in $B_{\rho}$ with $\eta = 1$ in $B_\frac\rho2$, and let
\beno
P_1=\int_{\mathbb{R}^3}\frac1{4\pi|x-y|}\left(\partial_i\partial_j[(u_i-(u_i)_\rho)(u_j-(u_j)_\rho)\eta]+\nabla\cdot(n\nabla\phi\eta)\right)(y,t)dy,
\eeno
where $(u)_\rho$ denotes the mean value of $u $ in $B_\rho$.  Set $P_2 = P - P_1$. Obviously,  $\Delta P_2 = 0$ in $B_{\frac\rho2}$.
Using the Calderon-Zygmund estimate and Riesz potential estimate, we have
\ben\label{ine:CZ estimate p 1 1}
\int_{B_{\rho}}|P_1|^{\frac32}dx\leq C \int_{B_{\rho}}|u-(u)_{B_\rho}|^3+C \rho^{\frac 34}\left(\int_{B_{\rho}}|n\nabla \phi|^\frac65dx\right)^{\frac54}.
\een
Fix $\rho = 1$,
we introduce a new cut-off function $\xi_\ell(x) = \eta(\frac{x}{r_{\ell-1} })$ and $\xi_1 = \eta(x)$. Then by $1 = \sum_{\ell=0}^k (\xi_\ell - \xi_{\ell+1})+ \xi_{k+1}$ in $B_1$ for $1\leq k\leq N$, we have
\beno
r_{N+1}^{-4} \int_{Q_{r_{N+1}}} |P-\bar P|^\frac32 &\leq& r_{N+1}^{-4} \int_{Q_{r_{N+1}}} |P_1-\bar{P}_1|^\frac32 + r_{N+1}^{-4} \int_{Q_{r_{N+1}}} |P_2-\bar{P}_2|^\frac32 \\
&:=& T_1 + T_2.
\eeno
For the term $T_2$, by the property of harmonic function in Lemma \ref{lem 1}, there holds
\beno\label{T2}\nonumber
T_2 &\leq& C r_{{N+1}}^{-4} r_{{N+1}}^{\frac32}\int_{Q_{r_{N+1}}} |\nabla P_2|^{\frac32} \\ \nonumber
&\leq& C r_{{N+1}}^{-\frac52} \frac{r_{N+1}^{3}}{({\rho}-r_{N+1})^{\frac92}}\int_{Q_{{\rho}}} |P_2-\bar{P}_2|^{\frac32}\\ \nonumber
&\leq& C r_{{N+1}}^{-\frac52} \frac{r_{N+1}^{3}}{({\rho}-r_{N+1})^{\frac92}}\int_{Q_{{\rho}}} |P-\bar{P}|^{\frac32} + C r_{{N+1}}^{-\frac52} \frac{r_{N+1}^{3}}{({\rho}-r_{N+1})^{\frac92}}\int_{Q_{{\rho}}} |P_1-\bar{P}_1|^{\frac32} \\
&\leq& C r_{{N+1}}^{\frac12} \int_{Q_{{\rho}}} |P-\bar{P}|^{\frac32}+ C r_{N+1}^\frac92 T_1.
\eeno
By $\eqref{eq:condition}$, we know that 
\ben\label{T2}
T_2 \leq C \varepsilon_0 + C r_{N+1}^\frac92 T_1.
\een
For the term $T_1$, since $1 = \sum_{\ell=0}^k (\xi_\ell - \xi_{\ell+1})+ \xi_{k+1}$, we have
\beno
T_1 &\leq& r_{N+1}^{-4} \int_{Q_{r_{N+1}}} \left|\int_{\mathbb{R}^3}\frac1{4\pi|x-y|}\left(\partial_i\partial_j\left[u_iu_j \left(\sum_{\ell=0}^k (\xi_\ell - \xi_{\ell+1})+ \xi_{k+1}\right) \eta\right]\right)(y,t)dy\right|^\frac32 \\
&& + r_{N+1}^{-4} \int_{Q_{r_{N+1}}} \left|\int_{\mathbb{R}^3}\frac1{4\pi|x-y|}\left(\nabla\cdot\left[n\nabla\phi \left(\sum_{\ell=0}^k (\xi_\ell - \xi_{\ell+1})+ \xi_{k+1}\right) \eta\right]\right)(y,t)dy\right|^\frac32 \\
&:=& T_{11} + T_{12}.
\eeno
By the support set of $\xi_\ell$, we have
\beno
T_{11} &\leq& r_{N+1}^{-4} \int_{Q_{r_{N+1}}} \left|\sum_{\ell=0}^{N-3} \int_{B_{r_{\ell-1}} \setminus B_{r_{\ell+1}}}\frac1{4\pi|x-y|}{\left(\partial_i\partial_j\left(u_iu_j (\xi_\ell - \xi_{\ell+1})\eta\right)\right)}(y,t)dy \right|^\frac32 \\
&&+ r_{N+1}^{-4} \int_{Q_{r_{N+1}}} \left|\int_{B_{r_{N-3}}}\frac1{4\pi|x-y|}{\left(\partial_i\partial_j\left(u_iu_j \xi_{N-2}\eta\right)\right)}(y,t)dy \right|^\frac32 \\
&:=& T_{111} + T_{112}.
\eeno
For the term $T_{111}$, since $(x,t) \in Q_{r_{N+1}}$ and $y \in B_{r_{\ell-1}} \setminus B_{r_{\ell+1}}$, $\ell = 0, 1,\cdots,N-3$, we know that
\beno
|x - y| \geq r_{\ell+3}.
\eeno
Then for the term $T_{111}$, we have
\beno
T_{111} &\leq& Cr_{N+1}^{-4} \int_{Q_{r_{N+1}}} \left(\sum_{\ell=0}^{N-3} r_{\ell+3}^{-3} \int_{B_{r_{\ell-1}}} |u|^2 dy\right)^\frac32 \\
&\leq& Cr_{N+1} \left(\sum_{\ell=0}^{N-3} r_{\ell+3}^{-3} \sup_t \int_{B_{r_{\ell-1}}} |u|^2\right)^\frac32.
\eeno
By $\eqref{ine:induction}$, we have
\beno
T_{111} &\leq& Cr_{N+1} \left(\sum_{\ell=0}^{N-3} C_1 \varepsilon_0^\frac12 \right)^\frac32 \leq CN 2^{-N} \left(C_1 \varepsilon_0^\frac12\right)^\frac32 \leq C \left(C_1 \varepsilon_0^\frac12\right)^\frac32.
\eeno
For the term $T_{112}$, by singular integral theorem, we have 
\beno
T_{112} 
&\leq& Cr_{N+1}^{-4} \int_{I_{r_{N+1}}} \int_{\mathbb{R}^3} ||u|^2 \xi_{N-2}\eta|^\frac32 \leq C r_{N+1}^{-4} \int_{Q_{r_{N-3}}} |u|^3.
\eeno
By $\eqref{ine:induction}$, we derive
\beno
T_{112} \leq Cr_{N+1}^{-4} r_{N-3}^5 \left(C_1 \varepsilon_0^\frac12\right)^\frac32 \leq C \left(C_1 \varepsilon_0^\frac12\right)^\frac32.
\eeno
Collecting $T_{111}$ and $T_{112}$, we have 
\beno
T_{11} \leq C \left(C_1 \varepsilon_0^\frac12\right)^\frac32.
\eeno
The estimate of term $T_{12}$ is same as the term $T_{11}$. By the support set of $\xi_\ell$, we have
\beno
T_{12} &\leq& r_{N+1}^{-4} \int_{Q_{r_{N+1}}} \left|\sum_{\ell=0}^{N-3} \int_{B_{r_{\ell-1}} \setminus B_{r_{\ell+1}}}\frac1{4\pi|x-y|}{\left(\nabla\cdot(n\nabla\phi \left(\xi_\ell - \xi_{\ell+1}\right) \eta)\right)}(y,t)dy \right|^\frac32 \\
&&+ r_{N+1}^{-4} \int_{Q_{r_{N+1}}} \left|\int_{B_{r_{N-3}}}\frac1{4\pi|x-y|}{\left(\nabla\cdot(n\nabla\phi \xi_{N-2} \eta)\right)}(y,t)dy \right|^\frac32 \\
&:=& T_{121} + T_{122}.
\eeno
For the term $T_{121}$, since $|x-y| \geq r_{\ell+3}$, we acquire
\beno
T_{121} &\leq& C\|\nabla \phi\|_{L^\infty}^\frac32 r_{N+1}^{-4} \int_{Q_{r_{N+1}}} \left(\sum_{\ell=0}^{N-3} \int_{B_{r_\ell}} r_{\ell+3}^{-2} n dy \right)^\frac32 \\
&\leq& C\|\nabla \phi\|_{L^\infty}^\frac32 r_{N+1} \left(\sum_{\ell=0}^{N-3} r_{\ell+3}^{-2} \sup_t \int_{B_{r_\ell}} n dy \right)^\frac32.
\eeno
By $\eqref{ine:induction}$, we know
\beno
T_{121} \leq C \|\nabla \phi\|_{L^\infty}^\frac32 \left(C_1 \varepsilon_0^\frac12\right)^\frac32.
\eeno
Similarly, by Riesz potential estimate in Lemma \ref{lem 3}, we achieve
\beno
T_{122} &\leq& Cr_{N+1}^{-4} \int_{I_{r_{N+1}}} \int_{B_{r_{N+1}}} \left|\int_{\mathbb{R}^3} |x-y|^{-2} \left|n\nabla\phi \xi_{N-2} \eta\right|(y,t)dy \right|^\frac32 \\
&\leq& Cr_{N+1}^{-\frac{13}4} \|\nabla \phi\|_{L^\infty}^\frac32 \int_{I_{r_{N-3}}} \left(\int_{B_{r_{N-3}}} n^\frac65\right)^\frac54.
\eeno
Noting that 
\beno
\int_{I_{r_{N-3}}} \left(\int_{B_{r_{N-3}}} n^\frac65\right)^\frac54 \leq C r_{N-3}^\frac54 \left(\int_{Q_{r_{N-3}}} n^\frac53 \right)^\frac9{10},
\eeno
we have 
\beno
T_{122} \leq Cr_{N+1}^{-\frac{13}4} \|\nabla \phi\|_{L^\infty}^\frac32 r_{N-2}^\frac54 r_{N-2}^\frac92 \left(C_1 \varepsilon_0^\frac12\right)^\frac32 \leq C \|\nabla \phi\|_{L^\infty}^\frac32 \left(C_1 \varepsilon_0^\frac12\right)^\frac32.
\eeno
Collecting $T_{121}$ and $T_{122}$, we get 
\beno
T_{12} \leq C \|\nabla \phi\|_{L^\infty}^\frac32 \left(C_1 \varepsilon_0^\frac12\right)^\frac32.
\eeno
The estimates of $T_{11}$ and $T_{12}$ implies that 
\ben\label{T1}
T_1 \leq C (1+\|\nabla \phi\|_{L^\infty}^\frac32) \left(C_1 \varepsilon_0^\frac12\right)^\frac32.
\een
By $\eqref{T1}$ and $\eqref{T2}$ we arrive 
\ben\label{Step2 con}
r_{N+1}^{-4} \int_{Q_{r_{N+1}}} |P-\bar P|^\frac32 \leq C (1+\|\nabla \phi\|_{L^\infty}^\frac32) \left(C_1 \varepsilon_0^\frac12\right)^\frac32 + C \varepsilon_0.
\een

{\bf Step 3: Estimate of  the term $r_{N+1}^{-3} \int_{B_{r_{N+1}}} n$.}

For the equation of $n$, multiplying $(\ref{eq:GKS})_1$ with $\psi$ and integration by parts on ${Q_1}$, we arrive
\ben\label{ine:energy n <x>}
\int_{B_1} (n \psi)(\cdot,t) = \int_{{Q_1}} n  (\partial_t \psi + \Delta \psi)
+ \int_{{Q_1}} n  u \cdot \nabla \psi+ \int_{{Q_1}} n \chi(c) \nabla c \cdot \nabla \psi.
\een
Using the property of $\psi$, we have
\begin{equation*}
\begin{aligned}
\int_{B_{r_{N+1}}} (n \psi)(\cdot,t)& ~\leq~ \int_{{Q_{r_3}}} n  (\partial_t \psi + \Delta \psi)
+ \int_{{Q_{r_3}}} n  u \cdot \nabla \psi+ \int_{{Q_{r_3}}} n \chi(c) \nabla c \cdot \nabla \psi\\&:=~ K_1+K_{2}+K_3.
\end{aligned}
\end{equation*}
For the term $K_1$, noting that $\partial_t \psi + \Delta \psi \leq C$, $\eqref{eq:condition}$ follows that
\beno
K_1 \leq C \int_{Q_{r_3}} n  \leq C \varepsilon_0.
\eeno
The estimates of $K_2$ and $ K_{3} $ are similar to $I_6$, direct calculations imply that
\beno
K_2 &\leq & \sum_{k=1}^N \int_{Q_{r_k} \setminus Q_{r_{k+1}}} nu\cdot \nabla \psi + \int_{Q_{r_{N+1}}}nu\cdot \nabla \psi \\
&\leq& C C_1^\frac32 \varepsilon_0^\frac34.
\eeno
and
\beno
K_3 &\leq & \sum_{k=1}^N \int_{Q_{r_k} \setminus Q_{r_{k+1}}}| n\chi(c)\nabla c \cdot \nabla \psi| + \int_{Q_{r_{N+1}}} |n\chi(c)\nabla c\cdot \nabla \psi| \\
&\leq& C \|\chi\|_0 \|c_0\|_{L^\infty}^\frac12 C_1^\frac32 \varepsilon_0^\frac34.
\eeno
Collecting the estimates of $K_1-K_3$, for any $t \in (-r_{N+1}^2,0)$, we have
\ben\label{Step3 con} \nonumber
\int_{B_{r_{N+1}}} (n\psi)(\cdot,t) &\leq& C\varepsilon_0 + C C_1^\frac32 \varepsilon_0^\frac34+C \|\chi\|_{0}\|c_{0}\|_{L^{\infty}}^{\frac12} C_1^\frac32 \varepsilon_0^\frac34 \\
&\le&  C \left(1+\|\chi\|_0 \|c_0\|_{L^\infty}^\frac12\right) C_1^\frac32 \varepsilon_0^\frac34 + C \varepsilon_0.
\een

{\bf Step 4: Estimate of  the term $r_{N+1}^{-3} \int_{B_{r_{N+1}}} n|\ln n|$.}

In order to prove (\ref{ine:induction}) for $k = N+1$, it's sufficient to estimate the term of  $r_{N+1}^{-3} \int_{B_{r_{N+1}}} (n + n |\ln n|)$.
Combining (\ref{Step1 con}) and (\ref{Step3 con}), for any $t \in (-r_{N+1}^2,0)$, we have
\ben\label{ine:induction 4}\nonumber
&&\int_{B_{r_{N+1}}} (n\psi)(\cdot,t)+\int_{B_{r_{N+1}}} (n \ln n\psi)(\cdot,t) + r_{N+1}^{-3}\int_{Q_{r_{N+1}}} |\nabla \sqrt{n}|^2\\ \nonumber
&& \quad + r_{N+1}^{-3} \int_{B_{r_{N+1}}} (|\nabla \sqrt{c}|^2 )(\cdot,t) + r_{N+1}^{-3}\int_{Q_{r_{N+1}}} |\Delta \sqrt{c}|^2\\ \nonumber
&& \quad + r_{N+1}^{-3}\int_{B_{r_{N+1}}}(|u|^2)(\cdot,t) + r_{N+1}^{-3}\int_{Q_{r_{N+1}}} |\nabla u|^2 + r_{N+1}^{-4} \int_{Q_{r_{N+1}}} |P-\bar P|^\frac32\\ 
&&\leq C( \Theta_0)(1+\|\chi\|_{0})(1+ \|\nabla \phi\|_{L^\infty}+ \|c_0\|_{L^\infty})^2C_1^{\frac{29}{18}}\varepsilon^{\frac7{12}}_0.
\een

Note that $|\ln n| n^{\frac{1}{30}}\leq 30e^{-1}$ for $0<n<1$, then by (\ref{Step3 con}) and (\ref{ine:induction 4}), we have
\ben\label{Step4 con} \nonumber
&&Cr_{N+1}^{-3} \int_{B_{r_{N+1}}} (n |\ln n|)(\cdot,t) dx \\ \nonumber
&\leq& \int_{B_{r_{N+1}}} (n\ln n \psi)(\cdot,t) dx -2\int_{B_{r_{N+1}}\cap \{x;0<n<1\}}(n\ln n \psi)(\cdot,t) dx \\ \nonumber
&\leq& \int_{B_{r_{N+1}}} (n\ln n \psi)(\cdot,t) dx + 60e^{-1} \int_{B_{r_{N+1}}}  (n^{\frac{29}{30}}\psi)(\cdot,t) dx \\ \nonumber
&\leq& C( \Theta_0)(1+\|\chi\|_{0})(1+ \|\nabla \phi\|_{L^\infty}+ \|c_0\|_{L^\infty})^2C_1^{\frac{29}{18}}\varepsilon^{\frac7{12}}_0+ \left(C C_1^\frac32 \varepsilon_0^\frac34 + C \varepsilon_0\right)^{\frac{29}{30}}\\
&\leq&C( \Theta_0)(1+\|\chi\|_{0})(1+ \|\nabla \phi\|_{L^\infty}+ \|c_0\|_{L^\infty})^2C_1^{\frac{29}{18}}\varepsilon^{\frac7{12}}_0,
\een
where we used the integral of heat kernel
\beno
\int_{B_{r_{N+1}}}  \psi dx\leq C.
\eeno

{\bf Step 5: Proof of the term $r_{N+1}^{-3}\int_{Q_{r_{N+1}}} |\nabla^2 \sqrt{c}|^2$.}

Let $\xi$ be a cut-off function, which equals $1$ on $Q_{r_{N+1}}$ and vanishes outside of $Q_{r_{N}}$.  Using integration by parts, we have
\beno\label{nabla2 sqrtc}\nonumber
&&r_{N+1}^{-3}\int_{Q_{r_{N+1}}} |\nabla^2 \sqrt{c}|^2 = r_{N+1}^{-3}\int_{Q_{r_{N+1}}} |\nabla^2 \sqrt{c}|^2\xi^2 \leq r_{N+1}^{-3}\int_{Q_{r_{N}}} |\nabla^2 \sqrt{c}|^2\xi^2 \\ \nonumber
&\leq& r_{N+1}^{-3}\left(\int_{Q_{r_{N}}}|\Delta\sqrt{c}|^2 \xi^2 +\int_{Q_{r_{N}}}\Delta \sqrt{c} \nabla \sqrt{c}\cdot \nabla\xi^2 - \int_{Q_{r_{N}}}\nabla^2\sqrt{c}:(\nabla \sqrt{c}\otimes\nabla\xi^2)\right),
\eeno
which means 
\beno
r_{N+1}^{-3}\int_{Q_{r_{N+1}}} |\nabla^2 \sqrt{c}|^2 
&\leq& Cr_{N+1}^{-3}\left(\int_{Q_{r_{N}}}|\Delta\sqrt{c}|^2 \xi^2 +  \int_{Q_{r_{N}}} |\nabla \sqrt{c}|^2 |\nabla\xi|^2 \right) \\
&:=&L_1+L_2.
\eeno
For the term of $L_1$, by $\eqref{Step1 con}$ and $\eqref{Step4 con}$, we have
\beno
L_1 \leq 8 r_N^{-3} \int_{Q_{r_{N}}}|\Delta\sqrt{c}|^2 .
\eeno
For the term of $L_2$, $\eqref{Step1 con}$ and $\eqref{Step4 con}$ imply that 
\beno
L_2 \leq C r_N^{-3} r_N^5 \sup_t \int_{B_{r_N}} |\nabla \sqrt c|^2.
\eeno
Collecting $L_1$, $L_2$ and by (\ref{Step4 con}), we have 
\ben\label{Step5 con}
r_{N+1}^{-3}\int_{Q_{r_{N+1}}} |\nabla^2 \sqrt{c}|^2 \leq C( \Theta_0)(1+\|\chi\|_{0})(1+ \|\nabla \phi\|_{L^\infty}+ \|c_0\|_{L^\infty})^2C_1^{\frac{29}{18}}\varepsilon^{\frac7{12}}_0.
\een 

{\bf Step 6: Proof of $\eqref{ine:induction}$ for $k = N+1$.}

Combining $\eqref{Step1 con}$, $\eqref{Step2 con}$, $\eqref{Step3 con}$, $\eqref{Step4 con}$ and $\eqref{Step5 con}$, we have 
\beno
&&r_{N+1}^{-3}\sup_{-r_{N+1}^2<t<0} \int_{B_{r_{N+1}}} n  + |n \ln n| + |\nabla \sqrt{c}|^2 + |u|^2 \\
&&+ r_{N+1}^{-3} \int_{Q_{r_{N+1}}} |\nabla \sqrt{n}|^2 + |\nabla^2 \sqrt{c}|^2 + |\nabla u|^2 + r_{N+1}^{-4} \int_{Q_{r_{N+1}}} |P-\bar P|^\frac32 \\
&&\leq C( \Theta_0)(1+\|\chi\|_{0})(1+ \|\nabla \phi\|_{L^\infty}+ \|c_0\|_{L^\infty})^2C_1^{\frac{29}{18}}\varepsilon^{\frac7{12}}_0.
\eeno
Choosing $\varepsilon_0= \frac{\varepsilon_1}{(1+\|\chi\|_{0})^{12}\left( 1+\|\nabla\phi\|_{L^{\infty}}+\|c_{0}\|_{L^{\infty}} \right)^{24}}$ and $\varepsilon_1$ depending $ \Theta_0 $ such that
\begin{equation*}
	C( \Theta_0)(1+\|\chi\|_{0})\left(1+\|\nabla\phi\|_{L^{\infty}}+\|c_{0}\|_{L^{\infty}}  \right)^{2}\varepsilon_0^{\frac{1}{12}}C_{1}^{\frac{11}{18}}\leq 1,
\end{equation*}
then we obtain
\beno
&&r_{N+1}^{-3}\sup_{-r_{N+1}^2<t<0} \int_{B_{r_{N+1}}} n  + |n \ln n| + |\nabla \sqrt{c}|^2 + |u|^2 \\
&&+ r_{N+1}^{-3} \int_{Q_{r_{N+1}}} |\nabla \sqrt{n}|^2 + |\nabla^2 \sqrt{c}|^2 + |\nabla u|^2 + r_{N+1}^{-4} \int_{Q_{r_{N+1}}} |P-\bar P|^\frac32 \\
&&\leq C_1 \varepsilon_0^\frac12.
\eeno

The proof of Proposition \ref{lem:induction} is complete.
\end{proof}

{\bf Proof of Theorem \ref{thm:1}.}
First, it is obvious  that Theorem \ref{thm:1}  under the condition $\eqref{eq:condition}$ follows from Proposition \ref{lem:induction}.

Next, we prove Theorem \ref{thm:1} under the condition $\eqref{eq:condition-32}$. 
Let $$ \varepsilon_2\leq \frac{\varepsilon_1^{2}}{(1+\|\chi\|_{0})^{20}\left( 1+\|\nabla\phi\|_{L^{\infty}}+\|c_{0}\|_{L^{\infty}} \right)^{40}}, $$
and the inequality 
\ben\nonumber
\int_{Q_{1}(z_0)} \left(n^{\frac32}(|\ln n|+1)^\frac32 + |\nabla \sqrt{c}|^3 + |u|^3+|P|^\frac32\right) \leq \varepsilon_2.
\een
comes true.

Recall the local energy inequality from (\ref{local energy inequality}),
letting $\psi$ is a cut-off function on domain $Q_{1}$, which means that $\psi = 1$ on $Q_{\frac12}$ and $\psi = 0$ outside $Q_{1}$. Using H\"{o}lder's inequality, there holds
\ben\label{2,1} \nonumber
&&\int_{B_\frac12} (n \ln n)(\cdot,t) + 4 \int_{Q_{\frac12}} |\nabla \sqrt{n}|^2+\frac{2}{\Theta_0}  \int_{B_\frac12} (|\nabla \sqrt{c}|^2)(\cdot,t)\\ \nonumber
&&+\frac{4}{3\Theta_0}\int_{Q_{\frac12}} |\Delta \sqrt{c}|^2 +\frac{18}{\Theta_0}\|{c_0}\|_{L^{\infty}}\int_{B_\frac12}(|u|^2)(\cdot,t) + \frac{18}{\Theta_0}\|{c_{0}}\|_{L^{\infty}}\int_{Q_{\frac12}} |\nabla u|^2\\ \nonumber
&\leq&  C(\Theta_0)(1+\|\chi\|_{0})\left(1+\|\nabla\phi\|_{L^{\infty}}+\|c_{0}\|_{L^{\infty}}\right)^2 \left(\|n\|_{L^\frac32(Q_1)}^\frac32 + \|n \ln n\|_{L^\frac32(Q_1)} + \|n \ln n\|_{L^\frac32(Q_1)}^\frac32  \right)\\ \nonumber
&&+  C(\Theta_0)(1+\|\chi\|_{0})\left(1+\|\nabla\phi\|_{L^{\infty}}+\|c_{0}\|_{L^{\infty}}\right)^2 \left( \|\nabla \sqrt c\|_{L^3(Q_1)}^2+\|\nabla \sqrt c\|_{L^3{(Q_1)}}^3 \right)\\\nonumber&&+  C(\Theta_0)(1+\|\chi\|_{0})\left(1+\|\nabla\phi\|_{L^{\infty}}+\|c_{0}\|_{L^{\infty}}\right)^2\left( \|u\|_{L^3(Q_1)}^2 + \|u\|_{L^3{(Q_1)}}^3 + \|P\|_{L^\frac32(Q_1)}^\frac32\right)\\
&\leq& C(\Theta_0)(1+\|\chi\|_{0})\left(1+\|\nabla\phi\|_{L^{\infty}}+\|c_{0}\|_{L^{\infty}}\right)^2\varepsilon_2^{\frac23}.
\een
Recall the $\eqref{ine:energy n <x>}$ as follow:
\beno
\int_{B_1} (n \zeta)(\cdot,t) = \int_{{Q_1}} n  (\partial_t \zeta + \Delta \zeta) + \int_{{Q_1}} n  u \cdot \nabla \zeta+ \int_{{Q_1}} n \chi(c) \nabla c \cdot \nabla \zeta.
\eeno
Using H\"{o}lder's inequality, there holds
\ben\label{2,2} \nonumber
\int_{B_\frac12} n &\leq& C (1+\|\chi\|_{0}) \left(\|n\|_{L^\frac32(Q_1)} + \|n\|_{L^\frac32(Q_1)}^\frac32 + \|\nabla \sqrt c\|_{L^{3}(Q_{1})}^3 + \|u\|_{L^3(Q_1)}^3\right) \\
&\leq& C(1+\|\chi\|_{0})\left(1+\|c_{0}\|_{L^{\infty}}^{\frac12}\right)\varepsilon_2^{\frac23}. 
\een
By $\eqref{2,1}$, we have 
\ben\label{2,3} \nonumber
\int_{B_\frac14} (n |\ln n|)(\cdot,t) dx &\leq& \int_{B_\frac12} (n\ln n \zeta)(\cdot,t) dx + 60e^{-1} \int_{B_\frac12}  (n^{\frac{29}{30}}\zeta)(\cdot,t) dx \\
&\leq& C(1 + \|\chi\|_0)\left(1+\|\nabla\phi\|_{L^{\infty}}+\|c_{0}\|_{L^{\infty}}\right)^2\varepsilon_2^{\frac{29}{45}}.
\een
Integrating by parts implies that 
\beno
\int_{Q_\frac14} |\nabla^2 \sqrt{c}|^2 \leq C \left(\int_{Q_\frac12}|\Delta\sqrt{c}|^2 \zeta^2 +\int_{Q_\frac12} |\Delta \sqrt{c}| |\nabla \sqrt{c}| |\nabla\zeta^2| + C \int_{Q_\frac12} |\nabla \sqrt{c}|^2 |\nabla\zeta|^2 \right).
\eeno
Using H\"{o}lder's inequality, there holds
\ben\label{2,4}
&&\int_{Q_\frac14} |\nabla^2 \sqrt{c}|^2 \leq C \left(\|\Delta \sqrt c\|_{L^2(Q_\frac12)}^2 + \|\nabla \sqrt c\|_{L^2(Q_\frac12)}^2\right)\nonumber\\
&& \leq C(1 + \|\chi\|_0)\left(1+\|\nabla\phi\|_{L^{\infty}}+\|c_{0}\|_{L^{\infty}}\right)^2\varepsilon_2^{\frac{29}{45}}.
\een
Collecting $\eqref{2,1}$, $\eqref{2,2}$, $\eqref{2,3}$, $\eqref{2,4}$ and $\eqref{eq:condition-32}$, there holds $\eqref{eq:condition}$ for some $\varepsilon_1$.

To sum up, we complete the proof of Theorem \ref{thm:1}.


\section{Appendix}
\begin{lemma}(See \cite{Lin})\label{mean value property}\label{lem 1}
	 Let $f$ be a harmonic function in $B_1 \subset \mathbb{R}^n$,
for $1\leq p,q \leq \infty$, $0<r<\rho<1$ and $k\geq1$, there holds
\beno
||\nabla^k f||_{L^q(B_r)}\leq C\frac{r^\frac{n}{q}}{(\rho-r)^{\frac np+k}}||f||_{L^p(B_\rho)}.
\eeno
\end{lemma}


\begin{lemma}(See Theorem 1 of Chapter 5 in \cite{Stein})\label{lem 3}
Assume $0<\alpha<n$, $1\leq p<\frac n{\alpha}$ and,
\beno
I_\alpha f(x)=\int_{\mathbb{R}^n}\frac{f(y)}{|x-y|^{n-\alpha}}dy
\eeno
when $p>1$ and $\frac1q=\frac1p-\frac\alpha n$, there holds
\beno
\|I_\alpha f\|_{L^q}\leq C\|f\|_{L^p}.
\eeno
\end{lemma}

 \end{document}